\documentclass[11pt]{amsart}
\usepackage{amsmath,amsthm, amscd, amssymb, amsfonts, mathrsfs,pb-diagram,pb-xy,mathtools}
\usepackage[nobysame]{amsrefs}

\usepackage{multicol}
\usepackage[all]{xy}
\usepackage[inline]{enumitem}
\usepackage{moreenum}
\usepackage{mathrsfs}
\usepackage[dvips, dvipsnames, usenames]{color}
\usepackage{hyperref}
\hypersetup{colorlinks=true,citecolor=cyan!70!black,linkcolor=red!60!black,linktocpage=true}
\usepackage{tikz-cd}
\usepackage[T2A,T1]{fontenc}
\usepackage{rotating}
\usepackage{multicol}
\usepackage{todonotes}

\usepackage{bbold}

\DeclareFontFamily{OT1}{pzc}{}
\DeclareFontShape{OT1}{pzc}{m}{it}{<-> s * [1.10] pzcmi7t}{}
\DeclareMathAlphabet{\mathpzc}{OT1}{pzc}{m}{it}

\allowdisplaybreaks

\newcommand{\ott}{\otimes_{\tau}}

\newcommand{\uno}{\mathbf{1}}

\newcommand{\ba}{ \mathbf{a}}

\newcommand{\Cgot}{\mathfrak C}
\newcommand{\Mgot}{\mathfrak M}
\newcommand{\Jgot}{\mathfrak J}

\newcommand{\At}{\mathtt{A}}
\newcommand{\Rt}{\mathtt{R}}
\newcommand{\Zt}{\mathtt{Z}}

\newcommand{\kut}{ \Bbbk^{\times}}

\makeatletter
\newcommand{\leqnomode}{\tagsleft@true}
\newcommand{\reqnomode}{\tagsleft@false}
\makeatother

\renewcommand{\_}[1]{_{\left( #1 \right)}}

\newcommand{\Dchaintwo}[3]{\xymatrix@C-4pt{\overset{#1}{\underset{\ }{\circ}}\ar
@{-}[r]^{#2}
& \overset{#3}{\underset{\ }{\circ}}}}

\newcommand{\Dchainfive}[9]{\xymatrix@C-6pt{\overset{#1}{\underset{\ }{\circ}}\ar
@{-}[r]^{#2}  & \overset{#3}{\underset{\ }{\circ}}\ar  @{-}[r]^{#4}  &
\overset{#5}{\underset{\ }{\circ}}
\ar  @{-}[r]^{#6}  & \overset{#7}{\underset{\ }{\circ}}\ar  @{-}[r]^{#8}  &
\overset{#9}{\underset{\ }{\circ}}}}

\newcommand{\Bg}{\mathfrak{B}}

\newcommand{\Ann}{\operatorname{Ann}}

\newcommand{\Der}{\operatorname{Der}}
\newcommand{\End}{\operatorname{End}}
\newcommand{\car}{\operatorname{char}}
\newcommand{\coh}{\operatorname{H}}
\newcommand{\hoch}{\operatorname{HH}}
\newcommand{\Jac}{\operatorname{Jac}}
\newcommand{\gr}{\operatorname{gr}}
\newcommand{\Hom}{\operatorname{Hom}}
\newcommand{\Ext}{\operatorname{Ext}}

\newcommand{\id}{\operatorname{id}} 
\newcommand{\im}{\operatorname{im}}
\newcommand{\lcm}{\operatorname{lcm}}
\newcommand{\ord}{\operatorname{ord}}
\newcommand{\op}{\operatorname{op}}

\newcommand{\GK}{\operatorname{GKdim}}

\newcommand{\Spec}{\operatorname{Spec}}
\newcommand{\REnd}{\operatorname{REnd}}
\newcommand{\RHom}{\operatorname{RHom}}

\newcommand{\lgo}{\mathfrak l}
\newcommand{\mgo}{\mathfrak m}

\newcommand{\ugo}{\mathfrak u}

\newcommand{\Rep}{\operatorname{Rep}}

\newcommand{\dgMod}{\text{-}\operatorname{dgmod}}

\newcommand{\gldim}{\operatorname{gldim}}
\newcommand{\projdim}{\operatorname{projdim}}
\newcommand{\ku}{ \Bbbk}
\newcommand{\fp}{{\mathbb F}_{\hspace{-2pt}p}}
\newcommand{\I}{\mathbb I}

\newcommand{\N}{\mathbb N}
\newcommand{\Z}{\mathbb Z}

\newcommand{\toba}{\mathscr{B}}

\newcommand{\bq}{\mathfrak{q}}

\newcommand{\Bc}{\mathcal B}
\newcommand{\Pc}{{\mathcal P}}

\newcommand{\Jc}{\mathcal{J}}

\newcommand{\Cc}{{\mathcal C}}

\newcommand{\Ss}{\mathcal{S}}

\newcommand{\yd}[1]{{}^{#1}_{#1}\mathcal{YD}}

\newcommand{\pf}{\begin{proof}}
\newcommand{\epf}{\end{proof}}

\makeatletter
\newtheorem*{rep@theorem}{\rep@title}
\newcommand{\newreptheorem}[2]{%
	\newenvironment{rep#1}[1]{%
		\def\rep@title{#2 \ref{##1}}%
		\begin{rep@theorem}}%
		{\end{rep@theorem}}}
\makeatother

\newtheorem{theorem}{Theorem}[section]
\newreptheorem{theorem}{Theorem}

\newtheorem{lemma}[theorem]{Lemma}

\newtheorem{cor}[theorem]{Corollary}

\newtheorem{prop}[theorem]{Proposition}

\theoremstyle{definition}
\newtheorem{definition}[theorem]{Definition}
\newtheorem{example}[theorem]{Example}

\newtheorem*{observation*}{Observation}
\newtheorem{question}[theorem]{Question}

\newtheorem*{claim*}{Claim}

\numberwithin{equation}{section}

\theoremstyle{remark}
\newtheorem{remark}[theorem]{Remark}

\newcommand{\ot}{\otimes}

\newcommand{\Vc}{\mathcal V}
\newcommand{\Vs}{\mathscr{V}}
\newcommand{\dg}{{\operatorname{dg}}}

\begin{document}

\noindent
\title[On the finite generation of the
cohomology of bosonizations]
{On the finite generation of the cohomology of bosonizations}

\author[Andruskiewitsch \emph{et. al.}]
{Nicol\'as Andruskiewitsch}
\address{FaMAF-CIEM (CONICET), Universidad Nacional de C\'ordoba,
Me\-dina A\-llen\-de s/n, \linebreak\indent  Ciudad Universitaria, 5000 C\' ordoba,  Argentina} 
\email{nicolas.andruskiewitsch@unc.edu.ar}

\author[]{David Jaklitsch}
\address{Department of Mathematics, University of Oslo,
Moltke Moes vei 35, 0851 Oslo, Norway} 
\email{dajak@math.uio.no}

\author[]{Van C. Nguyen}
\address{Department of Mathematics, United States Naval Academy, Annapolis, MD 21402, USA}
\email{vnguyen@usna.edu}

\author[]{Amrei Oswald}
\address{Department of Mathematics, University of Washington, Seattle, WA 98195, USA} 
\email{amreio@uw.edu}

\author[]{Julia Plavnik}
\address{Department of Mathematics, Indiana University, Bloomington, IN 47405, USA} 
\address{Department of Mathematics and Data Science, Vrije Universiteit Brussel, 1050 Brussels, Belgium}
\email{jplavnik@iu.edu}

\author[]{Anne V.\ Shepler}
\address{Department of Mathematics, University of North Texas,
Denton, Texas 76203, USA}
\email{anne.shepler@unt.edu}

\author[]{Xingting Wang}
\address{Louisiana State University, Baton Rouge, Louisiana, 70803, USA} 
\email{xingtingwang@math.lsu.edu}


\thanks{Date: \today.}

\begin{abstract}
We use deformation sequences of (Hopf) algebras, extending the results of Negron and Pevtsova, to show that bosonizations of some suitable braided Hopf algebras by some suitable finite-dimensional Hopf algebras have finitely generated cohomology. In fact, our results are shown in more generality for smash products. As applications, we prove the bosonizations of some Nichols algebras (such as  Nichols algebras of diagonal type, the restricted Jordan plane, Nichols algebras of 
direct sums of Jordan blocks plus points labeled with 1), 
by some suitable finite-dimensional Hopf algebras, have finitely generated cohomology, recovering some known results as well as providing new examples. 
\end{abstract}

\maketitle

\setcounter{tocdepth}{1}

\section{Introduction}\label{sec:problem}

\subsection{The context}
Let $\ku$ be a field, which we may assume  algebraically closed by  \cite{NWW2021}*{Lemma 3.1}. 
In this paper, we contribute to the Etingof--Ostrik conjecture on the finite generation
of  the cohomology rings of finite tensor categories \cite{etingof-ostrik03}, cf. also \cite{friedlander-suslin}. We shall say that a finite-dimensional Hopf algebra $H$ over $\ku$ has finitely generated cohomology, fgc for short, if 

\medbreak
\begin{enumerate}[leftmargin=7ex,label=\rm{(fgc\arabic*)}]
\item\label{item:fgca} the cohomology ring $\coh(H, \ku) = \bigoplus_{n\in \N_0} \Ext^n_H(\ku, \ku)$ is finitely generated, and 

\medbreak
\item\label{item:fgcb} 
$\coh(H,M) = \bigoplus_{n\in \N_0} \Ext^n_H(\ku, M)$ is a finitely generated $\coh(H, \ku)$-mo\-dule, for any finitely generated $H$-module $M$.
\end{enumerate}

For a finite tensor category $\Cc$, the fgc property can be formulated by replacing $\ku$ with the monoidal unit $\uno$.  
See \cite{aapw}*{Section 1.1} for background and references to the previous work.
Notice that for defining the cohomology ring, and thus the fgc property, it is enough to consider an augmented algebra, i.e., a $\ku$-algebra $R$ together with an algebra map $\epsilon_R\colon R\to\ku$.

\medbreak
One possible approach to the conjecture is to argue recursively, 
that is, to reduce it to smaller algebras that have fgc.
For example let us consider the bosonization $R\# K$ where $K$ is a finite-dimensional Hopf algebra 
and $R$ is a finite-dimensional Hopf algebra in the category $\yd{K}$  of Yetter-Drinfeld
modules over $K$, see \cite{radford-book}.
If the Hopf algebra $R\# K$ has fgc, then both $K$ and $R$ have fgc
by \cite{aapw}*{Theorem 3.2.1}.
It is natural to raise the following question.

\begin{question}
\label{question:bosonization} \cite{andruskiewitsch-natale}*{Section 4.1}
Let $K$ be a finite-dimensional Hopf algebra 
and $R$ a finite-dimensional Hopf algebra in $\yd{K}$.
If  $K$ and $R$  have fgc, does the bosonization $R\# K$ also have fgc?
\end{question}

Summarizing previous contributions and results in the present paper, we know now that Question \ref{question:bosonization} has a positive answer in the following cases:

\begin{enumerate}[leftmargin=7ex,label=\rm{(\alph*)}]
\item By \cites{aapw,andruskiewitsch-natale}, when $K = \ku G$, $G$ is a finite group and $\operatorname{char} \Bbbk$ does not divide $\vert G\vert$; 
and $R$ has fgc. See Theorem \ref{th:RtoRsmashH}.

\medbreak
\item By \cites{aapw,andruskiewitsch-natale}, when $K = \ku^G \coloneqq\Hom(\ku G, \ku)$, $G$ is a finite group; and $R$ has fgc. See Theorem \ref{th:RtoRsmashH}.

\medbreak
\item By \cite{negron}, when $K$ is cocommutative 
and $R$ is its $\ku$-dual $K^*$ where $K$ acts on $K^*$ by the coadjoint action.
Indeed, it is well-known that $K^* \# K\cong D(K)$, the Drinfeld double of $K$.

\medbreak
\item When $K$ is semisimple and $R$ admits a deformation sequence. See Theorem \ref{thm:K-semisimple-R-fgc}.

\medbreak
\item When $K$ is cocommutative and $R$ admits a $K$-equivariant deformation sequence. 
See Theorem \ref{theorem:negron}.

\medbreak
\item When $K$ admits a deformation sequence $\mathfrak C$ and $R$ admits a $\mathfrak C$-equivariant deformation sequence.  See Theorem \ref{thm:fgc-smash}. 
\end{enumerate}

\vspace{1ex}

\subsection{The main results and terminology}
In this paper, we contribute to Question \ref{question:bosonization}
 based on the concept of deformation sequence. 
Since several variations of  this notion are used,
we list them below. 
The terminology comes from \cite{negron},
although the meaning is slightly different.
We start with the notions discussed in Section \ref{sec:deformations}.

\begin{itemize}  [leftmargin=7ex]\renewcommand{\labelitemi}{$\circ$}
\item A \emph{deformation sequence} is a pair  of algebra maps
$Z \overset{\iota}{\hookrightarrow} Q \overset{\pi}{\twoheadrightarrow} R$
satisfying the seven conditions in Definition \ref{def:defor-seq}; particularly
$Q$ has finite global dimension and is module finite over the smooth central subalgebra $Z$, while $R$ is finite-dimensional. We say that $R$ admits a deformation sequence $Z \overset{\iota}{\hookrightarrow} Q \overset{\pi}{\twoheadrightarrow} R$ if such one exists. 

\medbreak
\item Closely related to the previous concept, 
a \emph{formal deformation sequence} is a pair of algebra maps
$\ku[[x_1,\cdots,x_n]]\overset{\iota}{\hookrightarrow}\widehat{Q}\overset{\pi}{\twoheadrightarrow} R$
as in Definition \ref{def:defor-seq-formal}.
\end{itemize}

Any deformation sequence gives rise to a formal one, see Lemma \ref{lem:formal-deformation}. Formal deformation sequences are instrumental to prove that algebras admitting a deformation sequence have fgc, 
a result from \cite{negron-pevtsova} summarized here in Theorem \ref{thm:negron-pevtsova}.  
Using ideas of the proof, we get our first result:

\begin{reptheorem}{thm:K-semisimple-R-fgc}
 Let $K$ be a semisimple Hopf algebra and 
$R$ a finite-dimensional augmented $K$-module algebra.
If $R$ admits a deformation sequence,
then the smash product $R \rtimes K$ has fgc. 
\end{reptheorem}

\medbreak
We then follow two slightly different paths from \cite{negron}. First, 
we fix a finite-dimensional Hopf algebra $K$ in Section \ref{sec:equivariant-deformation} and consider the following notion.

\medbreak
\begin{itemize} [leftmargin=7ex]\renewcommand{\labelitemi}{$\circ$}
\item  An \emph{equivariant deformation sequence} over a finite-dimensional Hopf algebra $K$ is a deformation sequence whose components are $K$-module algebras and the pair of maps in the sequence are $K$-linear.
\end{itemize}

\medbreak
By a careful analysis of the proof of \cite{negron}, 
we prove in Section \ref{sec:equivariant-deformation} our second result:

\begin{reptheorem}{theorem:negron}
If $K$ is a finite-dimensional cocommutative Hopf algebra and $R$ is an augmented $K$-module algebra that admits a $K$-equivariant deformation sequence, 
then the smash product $R\rtimes K$ has fgc. 
\end{reptheorem}

Second, we introduce in Section \ref{sec:deformations-hopf-alg} an alternative notion:

\medbreak
\begin{itemize} [leftmargin=7ex]\renewcommand{\labelitemi}{$\circ$}
\item A  \emph{ deformation sequence of Hopf algebras}
is a deformation sequence that is a central extension of Hopf algebras,
see Definition \ref{def:dfor-seq-hopf} for the precise formulation.
\end{itemize}

We then study the smash products of deformation sequences and prove our third result:
\begin{reptheorem}{thm:fgc-smash}
Let $K$ be a finite-dimensional Hopf algebra.
If $K$ admits a deformation sequence
$\Cgot$ of Hopf algebras
and $R$ is a finite-dimensional
algebra admitting
a $\Cgot$-equivariant deformation sequence, 
then the smash product $R \rtimes K$ has fgc.
\end{reptheorem}

\vspace{1ex}

\subsection{Applications}

We finish the paper with Section~\ref{sec:examples} to provide examples of Hopf algebras that satisfy the hypotheses of our results, consequently obtain fgc property for these algebras. 
Concretely, we first discuss in Subsection \ref{subsec:nichols-examples}
Nichols algebras as the class of braided Hopf algebras that
are suitable for the applicability of our results. 
Then we consider:

\begin{itemize} [leftmargin=7ex]\renewcommand{\labelitemi}{$\diamond$}
\item Quantum lines, a simple example to illustrate our techniques. 
See Subsection \ref{subsec:equivariant-quantum line}.

\medbreak
\item Quantum linear spaces, the next level of complexity. See Subsection \ref{subsec:equivariant-qls}. 

\medbreak
\item Nichols algebras of diagonal type. We show that the techniques of the present paper apply to Cartan type (up to a small technical condition) but do not seem to work for the other types.
See Subsection \ref{subsec:equivariant-cartan type}.

\medbreak
\item The restricted Jordan plane. See Subsection \ref{subsec:jordan}. 

\medbreak
\item A class of Nichols algebras introduced in \cite{aah-oddchar} that have fgc by \cite{andruskiewitsch-natale}. See Subsection \ref{subsec:blocks-points}.
\end{itemize}
 
For all the Nichols algebras sketched above, we recover bosonizations known to have fgc (most of the time with new proofs) and obtain new examples of bosonizations having fgc.

\medbreak
We close this introduction by pointing out that since we are addressing readers with expertise in Hopf algebras or in homological algebra, we included details
of a few well-known definitions or proofs to reach both communities.

\vspace{1ex}

\subsection{Notations and conventions}
Throughout, let $\ku$ be a base field of any characteristic, unless specified. We take all tensor products over $\ku$ unless otherwise indicated. We abbreviate {\em algebra} for $\ku$-algebra and \emph{affine $\mathbb K$-algebra} for finitely generated algebra over a commutative ring $\mathbb K$. We say an algebra is {\em flat over a subalgebra} if it is flat as a module via (left) multiplication over that subalgebra. We say a graded algebra $A = \bigoplus_{n\in \N_0} A^n$ is {\em connected} if $A^0 \cong \ku$. We denote by $A^H$ the invariant subring of an algebra $A$ under a suitable action of a Hopf algebra $H$. 

For any Hopf algebra $H$ with coproduct $\Delta$, we use Sweedler's notation for the comultiplicative structure, with the summation notation being suppressed, writing $\Delta(h)=
h_{(1)}\ot h_{(2)} \in H \otimes H$ for any $h \in H$. We write $\epsilon_H$ for the counit of a Hopf algebra, or a coalgebra, $H$.

\section{Preliminaries} \label{sec:hopf} 

Our references for Hopf algebras and Nichols algebras  are 
\cite{radford-book} and \cite{andrus-leyva},
respectively. We recall some well-known facts. 

\subsection{Smash products and bosonizations}
Let $H$ be a Hopf algebra. An $H$-module algebra is an algebra $R$ which is also
a left $H$-module and satisfies
\begin{align*}
h \cdot (xy) &= (h\_{1} \cdot x)(h\_{2} \cdot y),&\text{for all } h\in H, \, x,y \in R.
\end{align*}

The smash product $R\rtimes H$   is the vector space $R \otimes H$ with the multiplication
\begin{align*}
(x \otimes h)  (y \otimes t) &= x(h\_{1} \cdot y)\otimes h\_{2}t, &\text{for all } h, t\in H, \, x,y \in R.
\end{align*}
We shall omit the tensor product symbol in the calculations below and record for later use a result whose proof is straightforward.

\begin{lemma}\label{lema:smash}
Let $\varphi: H \to K$ be a morphism of Hopf algebras,
$R$ an $H$-module algebra,  $S$ a $K$-module algebra 
and $\xi: R \to S$ a morphism of  algebras that satisfies
\begin{align*}
\xi(h \cdot x) &= \varphi(h) \cdot \xi(x),&\text{for all } h&\in H, \, x \in R.
\end{align*}
Then $\xi \rtimes \varphi: R\rtimes H \to S\rtimes K$, $xh \mapsto \xi(x)\varphi(h)$,
is an algebra map. \qed
\end{lemma}

\vspace{1ex}

An  $H$-module algebra $R$ is \emph{augmented} if the algebra $R$ has an augmentation
$\epsilon_{R}$ such that $\epsilon_{R}(h \cdot x) = \epsilon_H(h)\epsilon_{R}( x)$,
for $h \in H, r \in R$.

\begin{example}\label{exa:trivial-action}
Recall that the action of $H$ on a module $M$ is trivial if $h \cdot x = \epsilon(h)x$ for all
$h\in H$ and $x \in M$. An (augmented) algebra $R$ becomes 
an (augmented) $H$-module algebra with the trivial action; any  
subalgebra of $R$ is an (augmented) $H$-module subalgebra.
\end{example}

\medbreak
Now let $R$ be a Hopf algebra in the category $\yd{H}$ of Yetter-Drinfeld modules over $H$.  Then the {\em bosonization}
$R \# H$ of $R$ by $H$ 
(aka Radford biproduct) is the vector space $R \otimes H$ with the smash product
multiplication and the smash product comultiplication; see 
\cite{radford-book}*{Chapter 11} for details. 

\begin{remark}
The bosonization $R \# H$
is a Hopf algebra, which contains $R$ as a subalgebra and $H$ as a Hopf subalgebra, provided with Hopf algebra maps 
$R \#H  \underset{\iota}{\overset{\pi}{\rightleftarrows}}H$ 
satisfying $\pi \iota = \id_H$. Thus, when $R$ and $H$ are both finite-dimensional, our results in this paper on the fgc property of $R \# H$ contribute to the Etingof--Ostrik conjecture.

Conversely, given a Hopf algebra $A$ and 
Hopf algebra maps $A \underset{\iota}{\overset{\pi}{\rightleftarrows}}H$ 
with $\pi \iota = \id_H$, it turns out that  $R = A^{\operatorname{co} \, \pi}$ is a  Hopf algebra in $\yd{H}$ and $A \cong R \# H$ as Hopf algebras.
\end{remark}

\begin{remark}
We use the symbol $\rtimes$ for general smash product and reserve the symbol $\#$ for the bosonization.
\end{remark}

\vspace{1ex}

\subsection{Extensions}
\label{Extensions}
Extensions of Hopf algebras were discussed by various authors in
various contexts. We follow the expositions in \cites{andrus-devoto,schneiderext}.
\begin{definition}
\label{ExactSequenceDef}
A sequence of morphisms of Hopf algebras
\begin{align*}
\xymatrix@C-10pt{\ku\ar  @{->}[r]^{} & W \ar  @{->}[r]^{\jmath }
&  H  \ar  @{->}[r]^{\wp } & K \ar  @{->}[r]^{}&  \ku },
\end{align*}
is \emph{exact}, in which case we say that $H$ is an {\em extension} of $K$ by $W$,
if 

\begin{enumerate}[leftmargin=7ex,label=\rm{(\alph*)}]
\item\label{item:suc-exacta-1} $\jmath$ is injective, and we identify $W$ with $\jmath(W)$,

\medbreak
\item\label{item:suc-exacta-2} 
$\wp$ is surjective,

\medbreak
\item\label{item:suc-exacta-3} 
$\ker \wp = H\, W^+$,
for $W^+=\ker \epsilon_W$
and

\medbreak
\item\label{item:suc-exacta-4} $W = H^{\operatorname{co} \wp}$.
\end{enumerate}
\end{definition}

\vspace{1ex}

We often abbreviate
$W \overset{\jmath}{\hookrightarrow} H \overset{\wp}{\twoheadrightarrow} K$
for an exact sequence $\xymatrix@C-10pt{
\ku\ar  @{->}[r]^{} & W \ar  @{->}[r]^{\jmath }&  H \ar  @{->}[r]^{\wp } 
& K \ar  @{->}[r]^{}&  \ku}$.

\begin{remark}\label{rem:extension-fflat}
Let $W \xhookrightarrow[]{\jmath} H$ be an injective morphism of Hopf algebras
such that $H$ is a faithfully flat $W$-module via $\jmath$ and 
$\jmath(W)$ is stable by the left adjoint action of $H$.
Define $K$ by \ref{item:suc-exacta-3}, i.e., $K = H/ H\jmath(W)^+$ and let $H\overset{\wp}{\twoheadrightarrow} K$
be the natural projection.
Then \ref{item:suc-exacta-4} holds, 
 we have an exact sequence $W \overset{\jmath}{\hookrightarrow} H \overset{\wp}{\twoheadrightarrow} K$, and $H$ is a faithfully coflat $K$-comodule via $\wp$,
see \cite{andrus-devoto}*{Corollaries 1.2.5 and 1.2.14} and \cite{schneiderext}.
\end{remark}

\begin{remark}\label{rem:extension-fcoflat}
Let  $H\overset{\wp}{\twoheadrightarrow} K$
be a surjective morphism of Hopf algebras
such that $H$ is a faithfully coflat $K$-comodule via $\wp$.
Define $W$ by \ref{item:suc-exacta-4}, i.e., $W = H^{\operatorname{co} \wp}$ and let 
 $W \xhookrightarrow[]{\jmath} H$ be the natural inclusion, $\jmath(W)$ is stable by the left adjoint action of $H$.
Then \ref{item:suc-exacta-3} holds, 
and we have an exact sequence $W \overset{\jmath}{\hookrightarrow} H \overset{\wp}{\twoheadrightarrow} K$,
see \cite{andrus-devoto}*{Corollary 1.2.5 and 1.2.14} and \cite{schneiderext}.
\end{remark}

\vspace{1ex}

The following result is useful to verify  faithful flatness over Hopf subalgebras.

\begin{remark}
Let $H$ be a Hopf algebra with bijective antipode and $W \subset H$ a (not necessarily normal) Hopf subalgebra of $H$. Then the following conditions are equivalent, cf. \cite{schneider-normal}*{Corollary 1.8}:
\begin{enumerate}[leftmargin=7ex]
\item\label{item:fflat-1} $H$ is faithfully flat as a left $W$-module.

\medbreak
\item\label{item:fflat-2} $H$ is projective as a left $W$-module.

\medbreak
\item\label{item:fflat-3} $H$ is a flat left $W$-module 
and $W$ is a direct summand of $H$, as a left $W$-module.

\medbreak
\item\label{item:fflat-4} Any of the conditions above with right instead of left.
\end{enumerate}
\end{remark}

\vspace{1ex}

\subsection{Finite generation and Noetherianity} \label{subsec:fgc}
Here we remind a reformulation of the fgc property.
We start by a well-known observation. If $A = \bigoplus_{n \in \N_0} A^n$ is a graded algebra,
then we set $A^{\text{even}} \coloneqq \bigoplus_{n \in \N_0} A^{2n}$. We say that $A$ is \emph{graded-commutative} if $xy = (-1)^{nm}yx$ whenever $x \in A^n$, $y \in A^m$; graded-centrality is defined accordingly. 

\begin{lemma}\label{lema:graded-Noetherian}
Let $A = \bigoplus_{n \in \N_0} A^n$ be a graded algebra such that the subalgebra $A^0$ is commutative and  Noetherian. 
\begin{enumerate}[leftmargin=7ex,label=\rm{(\roman*)}]
\item\label{item:A-Noeth-affine} If $A$ is graded-commutative, then the following are equivalent:

\smallbreak
\begin{enumerate}[label=\rm{(\alph*)}]
\item $A$ is Noetherian,

\smallbreak
\item $A$ is $A^0$-affine,

\smallbreak
\item $A$ is a finite $A^{\emph{even}}$-module, and $A^{\emph{even}}$ is $A^0$-affine.   
\end{enumerate}

\medbreak
\item Let $Z$ be a graded-central subalgebra of $A$ with
$Z^0= A^0$. If $A$ is Noetherian and a finite $Z$-module, then $Z$ is $Z^0$-affine and Noetherian.
\item \label{item:A-f.g.}If $A$ is left Noetherian, then $A$ is $A^0$-affine.
\end{enumerate}
\end{lemma}

\pf For (i) and (ii), see arguments and references in \cite{NWW2021}*{Proposition 2.4}. For (iii), see e.g., \cite{atiyah-macdonald}*{Proposition 10.7} or \cite{jia-zhang}*{Theorem 1.2}. 
\epf

The previous lemma supports the following reformulation of fgc stated in \cite{NWW2021}, taking into account the well-known fact that the cohomology ring $\coh(H, \ku)$ of any finite-dimensional Hopf algebra $H$ is graded-commutative, 
see e.g., \cite{suarez-alvarez}.

\begin{lemma}\label{lem:fgc=hfg} \cite{NWW2021}*{Proposition 2.9} A finite-dimensional Hopf algebra 
$H$ has fgc if and only if $\coh(H, M)$ is a Noetherian $\coh(H,\ku)$-module 
for any finite-dimensional $H$-module $M$. \qed
\end{lemma}

Finer characterizations of the fgc property will be discussed in Subsection \ref{subsec:alternatives-fgc}.

\vspace{1ex}

\subsection{Bosonizations with semisimple Hopf algebras}\label{subsec:bosonizations-fgc}
Here we recall some positive results of a particular case of Question \ref{question:bosonization}. 
Let $K$ be a semisimple Hopf algebra and $R$ a finite-dimensional Hopf algebra in $\yd{K}$ that has fgc. 
 It  seems to be  open whether $R\# K$ has fgc
except in the  cases considered in the next theorem.
Let $G$ be a finite group. 
For $K = \ku G$, this is 
\cite{aapw}*{Theorem 3.1.6} assuming  $\car \ku =0$
(but it is only needed that   $\ku G$ is semisimple);
as observed in \cite{andruskiewitsch-natale}, the same proof applies for
$K = \ku^G$.

\begin{theorem} \label{th:RtoRsmashH} 
Let $G$ be a finite group. Let $K$ be either $\ku G$ assumed semisimple, or else $\ku^G$.
If $R$ is a finite-dimensional Hopf algebra in $\yd{K}$ that has fgc, 
then the bosonization $R\# K$ has fgc.
\end{theorem}

To trace the necessity of the hypothesis, we recall the lemmas that lead to the proof.
The first step is the following result.

\begin{lemma} \label{lemma:SV}
Let $K$ be a semisimple Hopf algebra. Let $R$ be an augmented finite-dimensional $K$-module algebra 
and
let $M$ be an $(R \rtimes K)$-module.
Then  
\begin{align*}
\coh(R \rtimes K, \ku) &\cong \coh(R, \ku)^K, & 
\coh(R\rtimes K,M) &\cong  \coh (R,M)^K,
\end{align*} 
and the action of 
$\coh(R \rtimes K , \ku)$ on $\coh(R\rtimes K,M)$ is induced
by that of $\coh(R, \ku)$ on $\coh(R,M)$. 
\end{lemma}

\begin{proof}
This follows using the Lyndon-Hochschild-Serre spectral sequence, see
for example~\cite{stefan-vay}*{Theorem 2.17} or \cite{evens}*{Sections 7.2 and 7.3}
for the case when $K$ is a group algebra. 
\end{proof}

We will combine the previous lemma with 
the next result, a variation of the classical argument by Hilbert, cf. \cite{aapw}*{Lemma 3.1.1}.

\begin{lemma}\label{lemma:hilbert-invariants}
Let $K$ be a semisimple Hopf algebra. 
Let $A = \bigoplus_{n\in \N_0} A^n$ be a connected 
graded $K$-module algebra that 
is  (right) Noetherian.
Let $M$ be a finitely generated $(A\rtimes K)$-module.
Then $A^K$ is finitely generated and $M^K$ is a finitely generated
$A^K$-module. \qed
\end{lemma}

 In order to apply Lemma \ref{lemma:hilbert-invariants} we need that
 $\coh(R, \ku)$ is Noetherian under the recursive hypothesis that it has fgc.
 This follows from Lemma \ref{lema:graded-Noetherian} when $\coh(R, \ku)$
 is graded-commutative (particularly when $R$ is a Hopf algebra) but
 in general just a weaker version holds.

\medbreak
Let $K$ be a 
Hopf algebra,
not necessarily semisimple.
An algebra $A$ in $\yd{K}$ is  
\emph{braided commutative} if the multiplication $m_A$ satisfies
$m_A = m_A c_{A, A}$; in such case, the invariant subring $A^K$ is central in $A$.

As in \cite{aapw}, we rephrase Corollary~3.13 of \cite{MPSW} noting that
$\coh(R, \ku)=\Ext_{R}(\ku,\ku)$ is isomorphic 
to $\hoch(R,\ku)= \Ext_{R\ot R^{\text{op}}} (R,\ku)$, see \cite{MPSW}*{Section~2.4}.

\begin{lemma} \cite{MPSW}*{Corollary 3.13} \label{lemma:braided-commutative}
Let $K$ be a (not necessarily semisimple)  Hopf algebra and let $R$ be a bialgebra in $\yd{K}$. Assume that either $K$
or $R$ is finite-dimensional. Then $\coh(R, \ku)$  is a  braided commutative graded algebra
in $\yd{K}$. \qed
\end{lemma}

Actually, the previous result also holds under weaker hypotheses, see \cite{coppola-solotar}.

\begin{lemma}\label{lemma:braidedcommut-Noetherian} \cites{aapw,andruskiewitsch-natale}
Let $G$ be a finite  group and 
let $A$ be a braided commutative algebra either in $\yd{\ku G}$ or in $\yd{\ku^G}$. 
If $A$ is finitely generated (as an algebra), then $A$ is Noetherian. 
\end{lemma}

\begin{proof}
See \cite{aapw} for $\ku G$ and  \cite{andruskiewitsch-natale} for $\ku^G$.
\end{proof}

We are ready to prove the theorem. Let $K$ be either $\ku G$ (assumed semisimple)
or $\ku^G$.

\vspace{1ex}

\noindent
\emph{Proof of Theorem} \ref{th:RtoRsmashH}. 
Lemma \ref{lemma:braidedcommut-Noetherian},  
Lemma \ref{lemma:braided-commutative} and the hypothesis imply that 
$\coh(R, \ku)$ is Noetherian. Then $\coh(R, \ku)^{K}$ is finitely generated by Lemma \ref{lemma:hilbert-invariants}.
We conclude from Lemma \ref{lemma:SV} that 
$\coh(R \# K, \ku) \cong \coh(R, \ku)^{K}$ is finitely generated, i.e., 
\ref{item:fgca} holds for $R \# K$.
The proof of \ref{item:fgcb}  for $R \# K$ is  similar to
the proof of the same fact in \cite{aapw}*{Theorem 3.1.6}. 
\qed

\vspace{1ex}

\medbreak
Let $K$  and  $R$ be as in the beginning of the subsection.
In order to show that $R\# K$ has fgc with the previous arguments,
we only need to extend Lemma \ref{lemma:braidedcommut-Noetherian}
to the general case. 

\begin{question}
Let $K$ be a semisimple Hopf algebra and let 
$A$ be a braided commutative algebra in $\yd{K}$, finitely generated as an algebra. 
Does it follow that $A$ is Noetherian?
\end{question}

See Theorem \ref{thm:K-semisimple-R-fgc} below for a partial answer to this question. We remark that \cite{NWW2021}*{Lemma 5.8} assumes this property for $\coh(R,\ku)$, to conclude that the smash product $R \rtimes K$ has Noetherian cohomology ring.
\vspace{1ex}

\subsection{Power reductivity} \label{subsec:power}
\vspace{1ex}

Let $\mathbb K$ be a commutative ring and let $G$ be a (flat affine) group scheme over $\mathbb K$.
The history of the following question goes back to the XIX century and to Hilbert's Fourteenth Problem:

\medbreak
Under which conditions is the ring of invariants $R^G$ affine for any affine commutative $\mathbb K$-algebra $R$ on which $G$ acts?

\medbreak In \cite{vdKallen-survey}, van der Kallen introduced the notion of \emph{power reductive} group schemes. Given a group scheme $G$,
he showed (under suitable hypotheses on $\mathbb K$) 
that $R^G$ is affine for any 
affine commutative $\mathbb K$-algebra $R$ on which $G$ acts if and only if $G$ is power reductive. 
This generalizes the work of Hilbert, Noether, Nagata, Mumford, Haboush, Seshadri, \dots In particular, a finite group scheme is power reductive by a result of Grothendieck; 
other proofs were offered by Schneider and Ferrer Santos, 
see \cite{montgomery-book}*{Theorem 4.2.1}.
Borrowing the terminology from \cite{vdKallen-survey}, we introduce the following notion
needed later.

\begin{definition}\label{def:power-reductive}
A finite-dimensional Hopf algebra $H$ is \emph{power reductive} 
if any commutative $H$-module algebra $R$ is integral over its invariant subring $R^H$. 
In particular, if $R$ is affine, then $R$ is a finite module over $R^{H}$ and $R^H$ is affine
(by the Artin-Tate Lemma).
\end{definition}

As it stands, it appears that Definition \ref{def:power-reductive} is stronger than
asking that the invariant subring is affine; we wonder whether the converse is true,
namely if $R^H$ is affine, then necessarily $R$ is integral over $R^H$ (assuming that $H$ is finite-dimensional).

We summarize some classes of finite-dimensional Hopf algebras that are power reductive.

\begin{lemma} \label{HopfI}
A finite-dimensional Hopf algebra $H$ is power reductive 
in any of the following cases:

\begin{enumerate}[leftmargin=7ex,label=\rm{(\roman*)}]
\item\label{item:HopfI-poschar} 
the base field $\ku$ has positive characteristic,

\smallbreak 
\item \label{item:HopfI-ss} 
 $H$ is semisimple,

\smallbreak
\item \label{item:HopfI-coss}  $H$ is cosemisimple,

\smallbreak
\item\label{item:HopfI-cocom} $H$ is cocommutative. 
\end{enumerate}
\end{lemma}

\pf 
See \cite{skryabin-advances}*{Proposition 2.7}
for \ref{item:HopfI-poschar},
see \cite{skryabin-advances}*{Theorem 6.1} for
\ref{item:HopfI-ss}, 
see \cite{Zhu1996}*{Theorem 2.1} 
for
\ref{item:HopfI-coss},
and,
finally, note that
\ref{item:HopfI-cocom} was discussed above \cite{vdKallen-survey}.
\epf

See \cite{Zhu1996} for more examples.
Notice that Zhu \cite{Zhu1996} showed that
the four-dimensional Sweedler Hopf algebra, in characteristic zero, is not power reductive.

\vspace{1ex}

\section{Deformation sequences of algebras} \label{sec:deformations} 

In this section we review the notion of deformation sequence of algebras.
Theorem \ref{thm:negron-pevtsova} from  \cite{negron-pevtsova} gives a
criterion for finite generation of the cohomology of an augmented algebra.

\vspace{1ex}

\subsection{Deformation sequences} \label{subsec:deformations} 
Let us recall the main features of deformations as in
\cite{bezrukavnikov-ginzburg}, as set forth in \cite{negron-pevtsova}.
In the next definition,
we repeat the first three axioms
used to define exact sequences
of Hopf algebras 
(see Definition \ref{ExactSequenceDef})
and so enumerate them with (a), (b), (c) again.
By {\em smooth} in condition
\ref{item:deformation-Z}, we mean that the Zariski spectrum 
of the finitely generated $\ku$-algebra $Z$
has no singularities. 

\begin{definition}
\label{def:defor-seq}
A \emph{deformation sequence} is a pair  of algebra maps
$Z \overset{\iota}{\hookrightarrow} Q \overset{\pi}{\twoheadrightarrow} R$
such that
$Z$ has a distinguished point 
$\epsilon_Z: Z \to \ku$
satisfying the following seven conditions:

\begin{enumerate}[leftmargin=7ex,label=\rm{(\alph*)}]
\item\label{item:deformation-iota} $\iota$ is injective, in which case we identify $Z$ with $\iota(Z)$
in $Q$,

\medbreak
\item\label{item:deformation-pi} $\pi$ is surjective,

\medbreak
\item\label{item:ker-pi-Z} $\ker \pi = Q\, Z^+$ for $Z^+ = \ker  \epsilon_Z$,

\medbreak
\setcounter{enumi}{4}
\item\label{item:deformation-flat} $Q$ is finitely generated and flat as a 
left $Z$-module, 

\medbreak
\item\label{item:deformation-Z}  
$Z$ is smooth, finitely generated,
and central in $Q$,

\medbreak
\item\label{item:augmented-sequence} $Q$ and $R$ are augmented algebras, and
$\iota$ and $\pi$ preserve the augmentations, and

\medbreak
\setcounter{enumi}{7}
\item\label{item:finite-global-dim}  
$Q$ has finite global dimension.
\end{enumerate}
In this case, we say $R$ 
{\em admits
a deformation sequence}.
\end{definition}

\vspace{1ex}

\begin{remark}
\label{rem:Noetherian}
Note that in any
deformation sequence
$Z \hookrightarrow Q \twoheadrightarrow R$,
both $Q$ and $Z$ are Noetherian and finitely generated, $Q$ is PI and has finite Gelfand-Kirillov dimension. Indeed,
$Z$ is a finitely generated, commutative
algebra 
by Definition \ref{def:defor-seq} \ref{item:deformation-Z}
and thus Noetherian by the
Hilbert Basissatz. As $Q$ is finite over $Z$
by Definition \ref{def:defor-seq}
\ref{item:deformation-flat}, it is also
 Noetherian and finitely generated.
 Finally, $Q$ is PI by \cite{mcconnell-robson}*{Section 13.1.13, p. 481};
 hence $\GK Q < \infty$ by \cite{krause-lenagan}*{Corollary 10.7}.
\end{remark}

\vspace{1ex}

\begin{remark}\label{LastTermFiniteDimensional}
Given a deformation sequence 
$Z \hookrightarrow Q \twoheadrightarrow R$,
Definition \ref{def:defor-seq} \ref{item:deformation-flat} implies that
 the fiber  $\ku  \otimes_Z Q \cong R$ at the closed point $\epsilon_{Z}$ 
is finite-dimensional.
\end{remark}

\vspace{1ex}

\subsection{Formal deformation sequences}
\medbreak
There is a ``formal'' version of Definition \ref{def:defor-seq}:

\begin{definition}\label{def:defor-seq-formal} 
A pair of algebra maps 
$\widehat{Z} \overset{\widehat{\iota}}{\hookrightarrow} \widehat{Q} \overset{\widehat{\pi}}{\twoheadrightarrow} R$
is a \emph{formal deformation sequence} if  
the conditions of Definition \ref{def:defor-seq}
hold with \ref{item:deformation-Z} replaced by 

\medbreak
\begin{enumerate}[leftmargin=7ex,label=\rm{(\alph*')}]  
\setcounter{enumi}{5}
\item 
\label{item:deformation-local}  
$\widehat{Z}$ is central in $\widehat{Q}$ 
with $\widehat{Z} \cong   \ku [[y_1, . . . , y_n]]$,
formal power series in $n$ coordinates.
\end{enumerate}
\end{definition}

\vspace{1ex}

We verify below that every deformation sequence 
can be completed to a formal one,
a fact used to prove Theorem \ref{thm:negron-pevtsova} 
establishing the fgc for augmented algebra admitting a deformation sequence. Recall that an algebra $\At$ is  \emph{semilocal} if $\At/ \Jac \At$ is Artinian, where $\Jac \At$ is its Jacobson radical. We now recall a well-known fact on semilocal algebras
whose proof we include for completeness.

\begin{lemma}\label{lemma:semilocal}
Let $\At$ be an algebra finitely generated as a module 
over a central local subalgebra $\Zt$ with maximal ideal $\mathfrak{n}$.
\begin{enumerate}[leftmargin=7ex,label=\rm{(\roman*)}]  
\item\label{item:semilocal-1} $\At$ is semilocal.

\item\label{item:semilocal-2} Let $\Rt = \At / \mathfrak{n} \At$. Then $\At / \Jac \At \cong \Rt / \Jac \Rt$.
\end{enumerate}
 \end{lemma}

\pf Let $V$ be a simple $\At$-module. Pick $v \in V \backslash 0$ and set 
$I \coloneqq \Ann v \supseteq \Ann V$, where Ann  stands for the annihilator.
Thus $I$ is a maximal left ideal and $V \cong \At/I$. 
We claim that  $ J \coloneqq I \cap \Zt = \Ann V\cap \Zt$. 

Here $\supseteq$ is clear; for $\subseteq$ we argue as follows.
Pick  $x\in  I \cap \Zt$. Since $x \in I$, $x v = 0$; since $x \in \Zt$,
$xav = axv = 0$ for any $a \in \At$ hence $x \in \Ann \At v = \Ann V$.
 
Observe that the ideal $J$ of $\Zt$ is proper, otherwise $1 \in J$ and 
$V$ would be $0$; thus $J \subset \mathfrak{n}$.
We claim that $J = \mathfrak{n}$.

If $\mathfrak{n} \At \subset I$, then $\mathfrak{n} \subset I \cap \Zt = J$, hence $J = \mathfrak{n}$.
So, suppose that $\mathfrak{n} \At$ is not contained in the maximal left ideal $I$; 
hence  $\mathfrak{n} \At + I = \At$ and a fortiori $\mathfrak{n} V = V$. 
Now the $\Zt$-module $V$ is finitely generated by hypothesis; thus
the Nakayama Lemma \cite{atiyah-macdonald}*{Proposition 2.6} implies that $V = 0$, a contradiction.

Therefore, $\Jac \At \cap \Zt = \bigcap_{V \text{ simple}} \Ann V \cap \Zt = \mathfrak{n}$.
Hence $\At / \Jac \At$ is a finitely generated vector space over 
the field $\Zt / \mathfrak{n}$; thus $\At / \Jac \At$ is Artinian, i.e., $\At$ is semilocal
showing \ref{item:semilocal-1}.

\medbreak
Since $\mathfrak{n} = \Jac \At \cap \Zt$, 
$\mathfrak{n} \At \subseteq \Jac \At$, hence 
$\At / \Jac \At \cong (\At/\mathfrak{n} \At) / (\Jac \At / \mathfrak{n} \At) \cong \Rt / \Jac \Rt$, proving \ref{item:semilocal-2}.
\epf

In the next lemma, we view $\widehat{Z}\ot_Z Q$
as an algebra under multiplication
in each tensor component
recalling that, in any
deformation sequence $Z \hookrightarrow Q \twoheadrightarrow R$, we have $Z$ is central in $Q$.
We include a proof for
reader's convenience.
We use the notation $\mgo = Z^+$ noting $\mgo$ has a geometric flavor 
whilst $Z^+$ is more algebraic.

\begin{lemma}\label{lem:formal-deformation} \cite{negron-pevtsova}*{Lemma 2.10}
Any deformation sequence $Z \overset{\iota}{\hookrightarrow} Q \overset{\pi}{\twoheadrightarrow} R$
gives rise to a formal deformation sequence
\begin{equation}\label{eq:formal-deformation}
    \widehat{Z} \overset{\widehat{\iota}}{\hookrightarrow} \widehat{Z} \otimes_{Z} Q  \overset{\widehat{\pi}}{\twoheadrightarrow} R
\end{equation}
where $\widehat{Z}$ is the completion of the localization $Z_{\mgo}$ with respect to the $\mgo$-adic topology.  
\end{lemma}

\begin{proof}
We first apply to $Z \overset{\iota}{\hookrightarrow} Q \overset{\pi}{\twoheadrightarrow} R$
the localization functor at $\mgo$ as $Z$-modules. Then we apply to $Z_{\mgo} \overset{\iota}{\hookrightarrow} Q_{\mgo} \overset{\pi}{\twoheadrightarrow} R$ the completion functor at $\mgo_{\mgo}$ as $Z_\mgo$-modules. Concretely, the completion $\widehat{Z}$ of $Z_\mgo$ at 
its maximal ideal $\mgo_\mgo = Z^+_\mgo$ is
$\widehat{Z} = \varprojlim_{n} Z_\mgo/(Z^+_\mgo)^n$.
By \cite{atiyah-macdonald}*{Propositions 3.3 and 10.12}, the functor from the category of 
$Z$-modules to that of $\widehat{Z}$-modules given by 
\begin{align*}
M \to\widehat{M} &\coloneqq \varprojlim_{n} M_\mgo/ (Z^+_\mgo)^n M_\mgo 
\end{align*}
is exact. 
We obtain a  sequence of algebra maps
\begin{equation}\label{eq:formal-deformation-pf}
    \widehat{Z} \overset{\widehat{\iota}}{\hookrightarrow} \widehat{Q} \overset{\widehat{\pi}}{\twoheadrightarrow} \widehat{R}
\end{equation}
satisfying conditions \ref{item:deformation-iota} and 
\ref{item:deformation-pi} of Definition \ref{def:defor-seq}.
Also,  $\widehat{R} \cong R$ as $\mgo$
annihilates $R$ since $Z\hookrightarrow Q \twoheadrightarrow R$
is a deformation sequence.
Hence, we may apply the functor $\widehat{(-)}$
to the exact sequence
$\xymatrix@C-10pt{0 \ar@{->} [r] & \ker \pi
\ar  @{->}[r] & 
Q \ar  @{->}[r]^{ \pi }&  R  \ar@{->} [r] &0}$
of left $Z$-modules to conclude that 
$\xymatrix@C-10pt{0 \ar@{->} [r] & \widehat{\ker \pi} 
\ar  @{->}[r] & 
\widehat{Q} \ar  @{->}[r]^{\widehat{\pi} }&  R  \ar@{->} [r] &0}$
is an exact sequence of left $\widehat{Z}$-modules. But $\widehat{\ker \pi} \cong \widehat{\mgo Q} 
\cong \widehat{\mgo}\widehat{Q}$, thus
\eqref{eq:formal-deformation-pf} satisfies 
Definition \ref{def:defor-seq} \ref{item:ker-pi-Z}.
Now $\widehat{Q}\cong \widehat{Z}\otimes_{Z_{\mgo}} Q_{\mgo}$ 
as left $\widehat{Z}$-modules
by \cite{atiyah-macdonald}*{Proposition 10.13}
since $Z_{\mgo}$ is Noetherian and $Q_{\mgo}$ is finitely generated, 
and this is also an isomorphism
of algebras
as $Z_{\mgo}$ is central in $Q_{\mgo}$.
So, the sequence of algebra maps
\eqref{eq:formal-deformation-pf} boils down to \eqref{eq:formal-deformation}.

\smallbreak
Clearly, $\widehat{Z}$ is a central subalgebra of $\widehat{Q}$ and 
the latter is a finitely generated module over the former.
Since $\widehat{Z}$ is Noetherian,  so is $\widehat{Q}$, which is also PI.
 Also, $\widehat{Z}$ is flat over $Z$, see \cite{atiyah-macdonald}*{Corollary 3.6 and Proposition 10.14}.
By a standard argument, see \cite{atiyah-macdonald}*{2.2}, 
$\widehat{Q} \cong \widehat{Z}\otimes_ZQ$ is flat over $Q$,
analogously, it  is also flat over $\widehat{Z}$ because 
$Q$ is flat over $Z$. Thus \eqref{eq:formal-deformation} satisfies 
  Definition \ref{def:defor-seq} \ref{item:deformation-flat}.

\smallbreak
Next, $\epsilon_Z: Z\to \ku$ is a nonsingular point in $\Spec Z$, hence
$\widehat{Z}\cong \ku[[x_1,\ldots,x_n]]$ for $n = \dim Z_{\mgo}$.
Thus \eqref{eq:formal-deformation} satisfies  Definition \ref{def:defor-seq} \ref{item:deformation-local}.

\smallbreak 
It is clear from construction that \eqref{eq:formal-deformation} satisfies Definition \ref{def:defor-seq} \ref{item:augmented-sequence}.

\smallbreak 
It remains to show that $\widehat{Q}$ has finite global
dimension.
Let $T\coloneqq R/\Jac R$.
First note that $\widehat{T}\cong T$
since $\mgo$ annihilates $R$.
Second observe that
$\widehat{T}
\cong \widehat{Z}\otimes_Z 
\widehat{T}$ as left $\widehat{Z}$-modules
again using \cite{atiyah-macdonald}*{Proposition 10.13}
(as $Z$ is Noetherian and $T$ is finitely generated).
Third, notice that
$
R\cong \im \widehat\pi \cong \widehat Q / \ker \widehat\pi \cong \widehat Q/\widehat Q
\widehat Z^+ $
as algebras
since
$ \widehat{Z} \overset{\widehat{\iota}}{\hookrightarrow} \widehat{Q} \overset{\widehat{\pi}}{\twoheadrightarrow} R$ 
satisfies \ref{item:deformation-iota},
\ref{item:deformation-pi}, and \ref{item:ker-pi-Z}
of Definition \ref{def:defor-seq},
and likewise $R\cong Q/QZ^+$
as $ Z \overset{\iota}{\hookrightarrow} Q \overset{\pi}{\twoheadrightarrow} R$ is a deformation sequence.
We use the induced 
bimodule structure on $T=R/\Jac R$
over $\widehat Q$
and also over $Q$.
One may then check that
the isomorphism of left $\widehat{Z}$-modules
$$
  \widehat{Q} \ot_Q T
  \cong (\widehat Z\ot_Z Q) \ot_Q T
  \cong \widehat Z \ot_Z T
  \cong \widehat T
$$
is compatible with the left $\widehat{Q}$-action
giving an isomorphism of left $\widehat{Q}$-modules
$$
\begin{aligned}
  \widehat{Q} \ot_Q (R/\Jac R)
  =
  \widehat{Q} \ot_Q T
   \cong \widehat T
  \cong T
  = R/\Jac R
  \, .
\end{aligned}
$$
Since $Q\rightarrow \widehat{Q}$ is a flat extension,
we obtain an inequality of flat dimensions
of left modules:
\begin{equation*}
\text{flatdim}_{\widehat Q} (R/\Jac R)
=
\text{flatdim}_{\widehat Q}(
\widehat Q \ot_Q
(R/\Jac R) 
)
\leq
\text{flatdim}_{Q}(R/\Jac R)
\, .
\end{equation*}%

To finish the argument that
$\gldim \widehat{Q}$ is finite,
we relate projective dimension to flat dimension 
of left modules and
use the fact that
$\widehat{Q}$ is  PI  and Noetherian
and,
in addition, semilocal by Lemma
\ref{lemma:semilocal}\ref{item:semilocal-1}:
\begin{align*}
  \gldim \widehat{Q}
  &\ \overset{\clubsuit}{=} \
\projdim _{\widehat{Q}}(\widehat{Q}/\Jac \widehat{Q})
   \ \overset{\spadesuit}{=}\ 
\projdim _{\widehat{Q}}(R/\Jac R)
  \\
  &\ \overset{\diamondsuit}{=}
\ \text{flatdim}_{\widehat{Q}}(R/\Jac R)
   \ {\leq}\ 
  \text{flatdim}_{Q}(R/\Jac R)
  \\
  &\ \overset{\diamondsuit}{=} \ 
\projdim_{Q}(R/\Jac R)
\ \overset{\heartsuit}{\leq} \
\gldim Q \ <\ \infty.
\end{align*} 
Here 
\cite{WZ2000}*{Lemma 5.4 (2)}
implies the equality $\clubsuit$,  
Lemma \ref{lemma:semilocal} \ref{item:semilocal-2} implies 
$\spadesuit$,
\cite{Weibel}*{Proposition~4.1.5}
implies $\diamondsuit$
(see Remarks~\ref{rem:Noetherian} and \ref{LastTermFiniteDimensional}),
and the definition of global dimension
gives
$\heartsuit$.
Thus, \eqref{eq:formal-deformation} satisfies
Definition \ref{def:defor-seq} \ref{item:finite-global-dim} and hence is a formal deformation sequence.
\end{proof}

\vspace{1ex}

\subsection{Deformation sequences and fgc}
The following result is implicit in
\cite{negron-pevtsova}*{Corollary 4.7 and 
the proof of Theorem 4.8}. We sketch the argument for completeness.

\begin{theorem}\label{thm:negron-pevtsova} \cite{negron-pevtsova}
Any finite-dimensional augmented algebra $R$ admitting a deformation sequence has fgc.
\end{theorem}

\begin{proof}
By Lemma~\ref{lem:formal-deformation}, we may assume that 
$Z \hookrightarrow Q \twoheadrightarrow  R$ is a formal deformation sequence.
Let $\mgo = Z^+$. As in \cite{negron-pevtsova}, consider the algebras
\begin{itemize}  [leftmargin=7ex]\renewcommand{\labelitemi}{$\circ$}
\item $A_Z \coloneqq$ the symmetric algebra on the tangent space $\mgo/ \mgo^2$ shifted to degree 2;

\medbreak
\item $B_Z \coloneqq$ the exterior algebra on the cotangent space $\mgo/ \mgo^2$ 
shifted to degree $-1$.
\end{itemize}
By \cite{negron-pevtsova}*{Theorem 4.3},
the following are equivalent, for any finite $R$-modules $V$ and $W$:
\begin{multicols}{2}
\begin{enumerate}[leftmargin=7ex,label=\rm{(\Roman*)}]  
\item\label{item:ExtQ} $\Ext_Q (V, W)$ is finite over $B_Z$,

\item\label{item:ExtR} $\Ext_R (V, W)$ is finite over $A_Z$.
\end{enumerate}
\end{multicols}

Since $Q$ has finite global dimension, \ref{item:ExtQ} holds by 
\cite{negron-pevtsova}*{Corollary 4.7}, hence \ref{item:ExtR} holds.
Assume that  $R$ is augmented. Then $\Ext_R (\ku, \ku)$ is finite over $A_Z$,
which is a symmetric algebra in a finite number of variables, hence 
$\Ext_R (\ku, \ku)$ is a finitely generated algebra.
Given a finite $R$-module $V$, the action of $A_Z$ on 
$\Ext_R (\ku, V)$ factors through $\Ext_R (\ku, \ku)$. 
So finiteness over $A_Z$, which holds by \cite{negron-pevtsova}*{Theorem 4.3}, 
implies finiteness over $\Ext_R (\ku, \ku)$. This shows that $R$ has fgc. 
\end{proof} 

\vspace{1ex}

We obtain our first main result
giving the fgc for smash products
of semisimple Hopf algebras with algebras
admitting deformation sequences.

\begin{theorem}\label{thm:K-semisimple-R-fgc} Let $K$ be a semisimple Hopf algebra and 
$R$ a finite-dimensional augmented $K$-module algebra.
If $R$ admits a deformation sequence,
then the smash product $R \rtimes K$ has fgc. 
\end{theorem}

\pf 
Let $M$ be a finitely generated $(R \rtimes K)$-module. Since $R\rtimes K$ is finite-dimensional, $M$ is in particular finitely generated over $R$. Now, since $R$ admits a deformation sequence, $R$ has fgc by Theorem \ref{thm:negron-pevtsova}, and thus $\coh(R,M)$ is finitely generated over $\coh(R,\ku)$. Moreover, as explained in the proof of Theorem \ref{thm:negron-pevtsova}, $\coh(R,\ku)$ is a finitely generated module over a Noetherian algebra $A_Z$ and thus $\coh(R,\ku)$ is Noetherian as an algebra, as well. It follows that $\coh(R,M)$ is a Noetherian module over $\coh(R,\ku)$. By Lemma \ref{lemma:SV}, we have that 
\begin{align*}
\coh(R\rtimes K, M)&\cong \coh(R,M)^K &&\text{ and } &\coh(R\rtimes K,\ku)&\cong \coh(R,\ku)^K. 
\end{align*}
Then, Lemma \ref{lemma:hilbert-invariants} implies that $\coh(R\rtimes K, M)$ is finitely generated over $\coh(R\rtimes K,\ku)$.
 \epf

The following result is a straightforward consequence of Theorem~\ref{thm:K-semisimple-R-fgc}. 

\begin{cor}
Let $K$ be a semisimple Hopf algebra and 
$R$ a finite-dimensional  Hopf algebra in $\yd{K}$.
If $R$ admits a deformation sequence,
then the bosonization $R \# K$ has fgc.
\end{cor}

\section{Equivariant deformation sequences}\label{sec:equivariant-deformation}
In this section, we will closely follow the approach in \cite{negron}, where instead of a finite group scheme $G$, we will consider a finite-dimensional Hopf algebra $K$. Some results will hold without the assumption of the cocommutativity of $K$.

\subsection{Equivariant deformation sequences} We look at deformation sequences that preserve the action of a finite-dimensional Hopf algebra $K$.

\begin{definition}\label{def:equiv-defor-seq} \cites{negron-pevtsova, negron}
Let $K$ be a finite-dimensional Hopf algebra. A deformation sequence 
$Z \overset{\iota}{\hookrightarrow} Q \overset{\pi}{\twoheadrightarrow}  R$ 
is \emph{$K$-equivariant}
if  $Z$, $Q$ and $R$ are $K$-module algebras and $\iota$ and $\pi$ 
are $K$-module algebra maps. 
In this case, we say that  $R$ admits a \emph{$K$-equivariant deformation sequence}.

Analogously, a formal deformation sequence 
$\widehat{Z} \overset{\widehat{\iota}}{\hookrightarrow} \widehat{Q} \overset{\widehat{\pi}}{\twoheadrightarrow}  R$ is \emph{$K$-equivariant}
if  $\widehat{Z}$, $\widehat{Q}$ and $R$ are $K$-module algebras and
$\iota$ and $\pi$ are $K$-module algebra maps. 
In this case, we say that  $R$ 
admits
an \emph{$K$-equivariant formal deformation sequence}.
\end{definition}

\begin{remark}
Instead of smoothness of $\Spec Z$ imposed in Definition \ref{def:defor-seq} \ref{item:deformation-Z},
it is just required in \cites{negron,negron-pevtsova}
that  $\epsilon_Z$ is a nonsingular point of $\Spec Z$. 
\end{remark}

The following lemma, an extension of Lemma \ref{lem:formal-deformation}, shows that 
we may turn every $K$-equivariant deformation sequence for $R$ into  a formal one.

\begin{lemma}\label{lem:formal-deformation-equivariant} \cite{negron}*{Proposition 4.3} 
Let $K$ be a finite-dimensional Hopf algebra.  Suppose that
$Z \overset{\iota}{\hookrightarrow} Q \overset{\pi}{\twoheadrightarrow}  R$
is a $K$-equivariant deformation sequence. Then 
\[\widehat{Z} \overset{\widehat{\iota}}{\hookrightarrow} \widehat{Q} \overset{\widehat{\pi}}{\twoheadrightarrow}  R\]
is a $K$-equivariant formal deformation sequence, 
where $\widehat{Z}$ is the completion of the localization $Z_{\mgo}$ with respect to the
$\mgo$-adic topology and $\widehat{Q} \coloneqq \widehat{Z} \otimes_{Z} Q$
for $\mgo = Z^+$.  
\end{lemma}

\begin{proof}
By Lemma \ref{lem:formal-deformation},
we need only check equivariance
of the deformation sequence given there.
Since $\epsilon_Z=  \epsilon_R\circ \pi\circ \iota: Z\to \ku$ is $K$-linear, the maximal ideal $\mgo$ is $K$-stable. Hence, both $\widehat{Z}$ and $\widehat{Q}$ are $K$-module algebras 
(cf. \cite{negron}*{Proposition 4.3}), and the corresponding algebra maps between them are also $K$-linear. In particular, the unique maximal ideal $\mgo=(x_1,\ldots,x_n)$ in $\widehat{Z}$ is $K$-invariant. 
\end{proof}

\vspace{1ex}

\subsection{dg $K$-module algebras}
Let $K$ be a finite-dimensional Hopf algebra and denote by $\Rep K$ the tensor category of $K$-modules.

We recall here some properties of dg $K$-module algebras and later associate them to an equivariant deformation sequence to show the fgc property. We refer to \cite{yekutieli-book} for generalities on differential graded (dg) algebras.

\begin{definition}
We define dg objects
carrying an action of the Hopf algebra $K$.
\begin{itemize}
\item[(i)] A {\em dg $K$-module algebra}
is a dg algebra $A$ internal to $\Rep K$, i.e., a dg algebra $A$ 
which is also a $K$-module algebra such that all the dg algebra structure maps are $K$-linear.
\item[(ii)]
A {\em dg $K$-module $M$ over $A$} is a dg module over $A$ internal to $\Rep K$, i.e., a dg module $M$ over $A$ that is a $K$-module for which the dg module structure maps on $M$ are $K$-linear.
\item[(iii)]
Let 
$A\dgMod(\Rep K)$ denote the category of dg $K$-modules over $A$.
\end{itemize}
\end{definition}

\begin{remark} For a dg $K$-module algebra $A$, the smash product
$A \rtimes K$ becomes an ordinary dg algebra by setting $K$ to be concentrated in degree zero 
with zero differentials. Correspondingly, 
a dg $K$-module over $A$ is an ordinary dg module over $A \rtimes K$. In summary, we have an equivalence $A\dgMod(\Rep K)\simeq (A\rtimes K)\dgMod$.
In view of this, we denote by $D_\dg(A \rtimes K)$ the derived category of dg $K$-modules over $A$.
\end{remark}

We say that two $K$-module dg algebras $A,B$ are \emph{homotopically isomorphic} 
(see \cite{negron})
if there is a zig-zag of $K$-linear dg algebra quasi-isomorphisms 
\[
A\xleftarrow{\sim }  S_1\xrightarrow{\sim}S_2\xrightarrow{\sim } \cdots \xleftarrow{\sim}S_n\xrightarrow{\sim } B.
\]
Any homotopy isomorphism between two dg $K$-module algebras $A$ and $B$ provides 
a triangulated equivalence between their corresponding derived categories of dg modules:
\[D_\dg(A \rtimes K) \overset{\simeq}{\rightarrow}D_{\dg}(B \rtimes K)\]

Let $S$ be a dg $K$-module algebra, and  $M=\bigoplus_{i\in \mathbb Z} M_i$,
$N=\bigoplus_{i\in \mathbb Z}N_i$ two dg $K$-modules over $S$. Recall that 
\begin{align*}
\Hom_S(M,N)=\bigoplus_{i\in \mathbb Z}\Hom_{\gr  S}(M,N[i])
\end{align*}
where $\Hom_{\gr  S}(M,N[i])=\{f\in \Hom_S(M,N)\,|\, f(M_t)=N_{t+i}\}$ is a complex with differentials 
given by $(df)(m)=df(m)-(-1)^{|f|}f(dm)$. 
It is easy to see (cf. \cite{KKZ2009}*{Lemma 5.8}) 
that $\Hom_S(M,N)$ becomes a graded right $K$-module via 
\begin{align}\label{Kaction}
(f\cdot h)(m)=\sum \Ss(h\_1) \cdot f(h\_2\cdot m)
\end{align}
for any $f\in \Hom_S(M,N)$, $h\in K$, and $m\in M$, where $\Ss$ is the antipode of $K$. This $K$-action on $\Hom_S(M,N)$ is different from the one defined in \cite{negron}*{Section 2.5}. However, when $K$ is cocommutative, we can apply the antipode of $K$ to make $\Hom_S(M,N)$ into a left $K$-module which coincides with the module structure in \cite{negron}*{Section 2.5}.

Moreover, the image
$\RHom_S(M,N)$ of the right derived functor of $\Hom_S$ lies in $D(K)$, the derived category of $\Rep K$ \cite{negron}*{Lemma 2.3}, and thus extends to a functor
\[
\RHom_S(-,-)\colon D_\dg(S\rtimes K)^{\text{op}}\times D_\dg(S\rtimes K)\rightarrow D(K).
\]
It is direct to check that the composition map
\[\Hom_S(P,N)\times \Hom_S(M,P)\to \Hom_S(M,N)\] 
is $K$-linear for any dg $K$-modules $M$, $P$, and $N$ over $S$. Consequently, $\REnd_S(M)$ is a dg $K$-module algebra. 

\begin{lemma}\cite{negron}*{Lemma 2.4} \label{DG}
Let $K$ be a finite-dimensional Hopf algebra, $S$ a dg $K$-module algebra and $M$ a dg $K$-module over $S$. Then we have that:
\begin{enumerate}[leftmargin=7ex,label=\rm{(\roman*)}]
\item The dg $K$-module algebra $\REnd_S(M)$ is well-defined up to homotopy isomorphism. 
\item If $M$ and $N$ are isomorphic in $D_\dg(S\rtimes K)$, then $\REnd_S(M)$ and $\REnd_S(N)$ are homotopically isomorphic. 
\end{enumerate}  
\end{lemma}

\begin{proof}
The statement in \cite{negron}*{Lemma 2.4} considers a finite group scheme $G$, however, the proof does not require cocommutativity. 
\end{proof}

\vspace{1ex}

\subsection{Constructing dg $K$-module algebras associated to a deformation sequence}\label{subsec:constructing}

This subsection summarizes the construction of certain $K$-equivariant dg algebras associated to an equivariant deformation sequence that were used in \cite{negron} to prove fgc for Drinfeld doubles of finite group schemes. This can be considered as a $K$-equivariant version of the dg algebra tools from \cite{negron-pevtsova} employed to show fgc for augmented algebras that admit a (non-equivariant) deformation sequence. We will rely on these tools in Section \ref{sec:fgc_smash} to prove the fgc property for the smash product $R\rtimes K$ when the augmented algebra $R$ admits a $K$-equivariant deformation sequence.

It is important to point out that in order to modify the dg resolution from \cite{negron-pevtsova} to the $K$-equivariant setting, we need $K$ to be cocommutative where its representation category has a symmetric braiding given by the usual flipping map.

Throughout this subsection we consider the following setting: let $K$ be a finite-dimensio\-nal cocommutative Hopf algebra, $R$ a finite-dimensional augmented $K$-module algebra and
\begin{align*}
Z\hookrightarrow Q\twoheadrightarrow R
\end{align*}
a $K$-equivariant formal deformation sequence such that $Z=\ku[[x_1,\cdots,x_n]]$ with $K$-inva\-riant maximal ideal $\mgo=(x_1,\ldots,x_n)$.
\begin{lemma}\cite{negron}*{Lemma 3.1 and Lemma 5.2}\label{dgK}
There is a dg $K$-module algebra $\mathcal K$ satisfying:
\begin{enumerate} [leftmargin=7ex,label=\rm{(\roman*)}]
\item $\mathcal K$ is a $Z$-algebra that is finite and flat over $Z$.

\medbreak
\item $\mathcal K$ is homotopically isomorphic to $\ku$ and $\ku\otimes_Z\mathcal K$ is homotopically isomorphic to $B_Z$, the exterior algebra on the cotangent space $\mgo/ \mgo^2$ 
shifted to degree $-1$ with zero differentials.

\medbreak
\item Consider $\mathcal K$ as a $Z$-central dg $\mathcal K$-bimodule in $D(\mathcal K\otimes_Z  \mathcal K^{\text{op}})$. Then the dg $K$-module algebra $\REnd_{\mathcal K\otimes_Z\mathcal K^{\op}}(\mathcal K)$ is homotopically isomorphic to $A_Z$, the symmetric algebra on the tangent space $\mgo/ \mgo^2$ shifted to degree $2$ with zero differentials.

\medbreak
\item The dg $K$-module algebra $\mathcal K_Q \coloneqq Q\otimes_Z\mathcal K$ is homotopically isomorphic to $R$ and finite and flat over $Z$. 

\end{enumerate}
\end{lemma}

\begin{proof}
The construction of the dg $K$-module algebra $\mathcal K$ and its properties stated in (i) and (ii) are provided in \cite{negron}*{Lemma 3.1}. 
Property (iii) is \cite{negron}*{Lemma 5.2}. For (iv), the construction of $\mathcal K_Q=Q\otimes_Z\mathcal K$ and the fact that it is homotopically isomorphic to $R$ are explained in \cite{negron}*{Section 5.1}. Since $Q$ is finite and flat over $Z$, it is clear that $\mathcal K_Q$ is finite and flat over $Z$ by (i). 
\end{proof}

Consider the dg $K$-module algebras $\mathcal K$ and $\mathcal K_Q=\mathcal K\otimes_Z Q$ from Lemma \ref{dgK} that are homotopically isomorphic to $\ku$ and $R$, respectively. In \cite{negron}*{Section 5.2}, the authors construct a functor 
\[ \mathfrak{def}^K: D_{\dg}((\mathcal K\otimes_Z\mathcal K^{\op})\rtimes K)
\rightarrow
D_{\dg}(\mathcal K_Q\rtimes K)\]
from the derived  category of dg $K$-modules over $\mathcal K\otimes_Z\mathcal K^{\op}$ to that of dg $K$-modules over $\mathcal K_Q$. In particular,  for any dg $K$-module $M$ over $\mathcal K_Q$, there is a dg $K$-module algebra map 
\begin{multline}\label{eq:action_of_S}
\mathfrak{def}^K_M: \quad \mathcal T\coloneqq\REnd_{\mathcal K\otimes_Z\mathcal K^{\op}}(\mathcal K)\xrightarrow{-\otimes^L_ZQ} \REnd_{\mathcal K_Q\otimes_Q\mathcal K^{\op}_Q}(M)
\\
\longrightarrow
\REnd_{\mathcal K_Q\otimes \mathcal K^{\op}_Q}(M)\xrightarrow{-\otimes^L_{\mathcal K_Q}M} \REnd_{\mathcal K_Q}(M).
\end{multline}
Moreover, for any dg $K$-modules $M$ and $N$ over $\mathcal K_Q$, the dg $K$-module algebra $\mathcal T$ acts on the dg $K$-module $\RHom_{\mathcal K_Q}(M,N)$ via either $\mathfrak{def}^K_M$ or $\mathfrak{def}^K_N$ with actions coinciding. The following summarizes the finite generation property derived from a $K$-equivariant formal deformation sequence, which is an equivariantization of 
\cite{negron-pevtsova}*{Corollary 4.7}.

\begin{lemma}\cite{negron}*{Theorem 5.4}\label{fgcR}
Retain the above notation. Let $M,N$ be two dg $K$-modules over $\mathcal K_Q$ such that $\coh(M)$ and $\coh(N)$ are finite-dimensional. Then
\begin{enumerate} [leftmargin=7ex,label=\rm{(\roman*)}]
\item $\coh(\mathfrak{def}^K_M): \coh(\mathcal T)\to \coh(\REnd_{\mathcal K_Q}(M))$ is a finite map of graded dg $K$-module algebras.
\item $\coh(\RHom_{\mathcal K_Q}(M,N))$ is a finitely generated module over $\coh(\mathcal T)$.
\end{enumerate}
\end{lemma}
\begin{proof}
Similarly to Lemma \ref{dgK}, the statement and proof appear in \cite{negron}*{Theorem 5.4} in terms of a finite group scheme $G$ instead of the cocommutative Hopf algebra $K$.
\end{proof}

\vspace{1ex}

\subsection{Alternative formulations of fgc}\label{subsec:alternatives-fgc}
In this subsection, $K$ is a finite-dimensional Hopf algebra.
Lemma \ref{lem:fgc=hfg} is the starting point for the following characterizations of the fgc property.
Notice that the equivalence of the first four items was established in \cite{NWW2021}*{Lemma 3.3} while the equivalence with the last item is implicit in \cite{negron}.

\begin{prop}\label{intP}
If $K$ is power reductive, then the following are equivalent:
\begin{enumerate}[leftmargin=7ex,label=\rm{(\roman*)}]
\item\label{item:alternative1} $K$ has fgc. 

\medbreak
\item\label{item:alternative2} $\coh(K, M)$ is a Noetherian $\coh(K, R)$-module, 
for any affine commutative $K$-module
algebra $R$ and any finitely generated $(R \rtimes K)$-module $M$.

\medbreak
\item\label{item:alternative3} $\coh(K, R)$ is finitely generated for any affine commutative $K$-module algebra $R$.

\medbreak
\item\label{item:alternative4} Let $T$ be a finitely generated non-negatively graded Noetherian $K$-module algebra
and let $M$ be a finitely generated $(T\rtimes K)$-module. 
Then $\coh(K, M)$ is Noetherian over $\coh(K, T)$
whenever $T$ is module finite over a graded central $K$-module subalgebra $Z$.

\medbreak
\item\label{item:alternative5} 
Let $S$ be a dg $K$-module algebra and $M$ be a dg $K$-module over $S$. 
Assume that

\smallbreak
\begin{enumerate}[label=\rm{(\alph*)}]
\item $\coh(S)$ is affine graded-commutative and concentrated in non-negative degrees,

\smallbreak
\item $\coh(M)$ is finitely generated over $\coh(S)$, and

\smallbreak
\item $S$ is homotopically isomorphic to its cohomology ring $\coh(S)$.
\end{enumerate}

\medbreak
\noindent
 Then the cohomology $\coh(\RHom_K(\ku, M))$ is Noetherian over 
 $\coh(\RHom_K(\ku, S))$.
\end{enumerate}
\end{prop}

\begin{proof}
The equivalences \ref{item:alternative1} $\iff$ \ref{item:alternative2}
$\iff$ \ref{item:alternative3} $\iff$ \ref{item:alternative4} are proved in \cite{NWW2021}*{Lemma 3.3}.
The delicate part
 is  \ref{item:alternative1} $\implies$ \ref{item:alternative2}; 
 we recap the other implications as follows.

\begin{itemize}
\item 
For \ref{item:alternative2} $\implies$ \ref{item:alternative3}, take $M = R$.

\smallbreak
\item
For
\ref{item:alternative3} $\implies$ \ref{item:alternative1}, take $R = \ku$ to 
see that \ref{item:fgca} 
holds and
then take $R = \ku \oplus M$ (with $M^2 = 0$) for a finite $K$-module $M$ to see
that \ref{item:fgcb} holds.

\smallbreak
\item
For
\ref{item:alternative2} $\implies$ \ref{item:alternative4},
take $R = Z^{\text{even}}$ 
to see that $\coh(K,M)$ is a Noetherian module
over $\coh(K, Z^{\text{even}})$.
Then $\coh(K,M)$ is a Noetherian module over $\coh(K,T)$
because the action of $\coh(K,T)$  on $\coh(K,M)$
restricts to the action of $\coh(K,Z^{\text{even}})$.

\smallbreak
\item
For
\ref{item:alternative4} $\implies$ \ref{item:alternative2},
take $T=R$ to be the affine commutative $K$-module algebra concentrated in degree zero.
\end{itemize}

\medbreak
We now turn to the new part of the lemma. 
For
\ref{item:alternative5} $\implies$ \ref{item:alternative1},
take $S$ to be the trivial dg algebra
$\ku$ concentrated in degree zero. Since 
$\coh(\RHom_K(\ku, \ku))\cong \Ext_K (\ku, \ku)$ the result follows by applying Lemma \ref{lem:fgc=hfg}.

We prove \ref{item:alternative4} 
$\implies$ \ref{item:alternative5}. 
By the assumptions in \ref{item:alternative5}, it is clear that both $\coh(S)$ and $\coh(M)$ are bounded below. Let $P_{\bullet}\to \ku$ 
be a projective resolution of $\ku$ as a $K$-module. 
This gives a double complex $\Hom_K(P_\bullet,M)$ whose cohomology yields a bounded spectral sequence
\begin{align*}
E_2^{p, q}(M)
=\Ext^{q}_K(\ku,\coh^p (M))
\implies E_\infty(M)=\coh^{p+q}(\RHom_K(\ku, M)).
\end{align*}

Take the finitely generated graded-commutative $K$-algebra $T=\coh(S)$ (which is also Noetherian by Lemma \ref{lema:graded-Noetherian} \ref{item:A-Noeth-affine}), and $N=\coh(M)$, which is a finitely generated module over $T$ and thus over $T\rtimes K$. 
We conclude by \ref{item:alternative4}
that $\coh(K,N)$ is Noetherian over $\coh(K,T)$ or, equivalently, 
that 
$E_2(M)=\coh(\RHom_K(\ku,\coh(M))$
 is Noetherian over $E_2(S)=\coh(\RHom_K(\ku,\coh(S))$. 
 
We claim that $E_2(S)\cong E_\infty(S)$. This follows from assumption \ref{item:alternative5}(c) 
that there is a homotopical isomorphism  between $S$ and $\coh(S)$, which induces an isomorphism of cohomology rings 
\[
E_2(S)=\coh(\RHom_K(\ku, \coh(S))
\ \cong \
\coh(\RHom_K(\ku, S))
=E_\infty(S).
\]

It remains to show $E_\infty(M)$ is Noetherian over $E_\infty(S)$, which follows from a classical argument of Friedlander and Suslin 
\cite{friedlander-suslin}*{Lemma 1.6}.
Indeed, since the spectral sequence $E_2(M)\implies E_\infty(M)$ is compatible with the action of the cohomology ring $E_2(S)\cong E_\infty(S)$, we have submodules over
$E_2(S)$
\[
B_2(M)\subseteq B_3(M)\subseteq \cdots \subseteq B_\infty(M)\subset Z_\infty(M)\subseteq \cdots \subseteq Z_3(M)\subseteq Z_2(M) 
\]
with $E_i(M)\cong Z_i(M)/B_i(M)$ for $i\ge 2$. 
We take quotients by $B_2(M)$.
As
$Z_{\infty}(M)/B_2(M)$
is a submodule of the
Noetherian $E_2(S)$-module $Z_2(M)/B_2(M)=E_2(M)$,
it is also
Noetherian,
and hence its quotient
$$
(Z_{\infty}(M)/B_2(M))
\ / \ 
(B_{\infty}(M)/B_2(M))
\ \cong\ 
Z_{\infty}(M)/B_{\infty}(M)
= E_{\infty}(M)
$$
is Noetherian as well.
\end{proof}

\vspace{1ex}

\subsection{The fgc property for smash products via equivariant deformation sequences}\label{sec:fgc_smash}

We can now extend the main result of \cite{negron} using the following lemma.

\begin{lemma}\cite{NWW2021}*{Appendix A3} \label{lem:NWW}
Let $K$ be a finite-dimensional Hopf algebra, $R$ a finite-dimensional augmented $K$-module algebra and $V$ a finitely generated $(R\rtimes K)$-module. Then
\begin{align*}
   \RHom_{K}(\ku, \RHom_R(\ku,V))\cong\RHom_{R\rtimes K}(\ku,V)\,. 
\end{align*}
\end{lemma}

\begin{proof}
We recall the spectral sequence related to $R\rtimes K$ discussed in \cite{NWW2021}*{Appendix A3}. It is clear that $\Hom_{R\rtimes K}(\ku,V)=\Hom_K(\ku,\Hom_R(\ku,V))$, where $K$ acts on $\Hom_R(\ku,V)$ by \eqref{Kaction}. In particular,  $\Hom_R(P,V)$ is a projective $K$-module for any projective $(R\rtimes K)$-module $P$ \cite{NWW2021}*{Lemma A.1}. By taking projective resolutions for $\ku$ as a $K$-module and as a $(R\rtimes K)$-module, we obtain a Lyndon–Hochschild–Serre spectral sequence 
\[
E_2^{p,q}(V)=\Ext^q_K(\ku, \Ext^p_R(\ku,V)) \implies
\Ext^{p+q}_{R\rtimes K}(\ku,V),
\]
which in terms of the corresponding right-derived functors translates to the desired result. 
\end{proof}

We now obtain our second main result on the fgc property of smash products:

\begin{theorem} \label{theorem:negron}
Let $K$ be a finite-dimensional cocommutative Hopf algebra and $R$ a finite-dimensional augmented $K$-module algebra. If $R$ admits a $K$-equivariant deformation sequence, then the smash product $R\rtimes K$ has fgc. 
\end{theorem}

\begin{proof}
We need to show that $\coh(R\rtimes K,\ku)$ is finitely generated and that for any finitely generated $(R\rtimes K)$-module $V$, we have that $\coh(R\rtimes K,V)$ is finitely generated over $\coh(R\rtimes K,\ku)$. Owing to Lemma \ref{lem:formal-deformation}, we can
consider a formal $K$-equivariant deformation sequence
\begin{align*}
Z=\ku[[x_1,\cdots,x_n]]\hookrightarrow Q\twoheadrightarrow R\,.
\end{align*}
By Theorem \ref{thm:negron-pevtsova}, we have that the augmented algebra $R$ has fgc.
Additionally, $K$ also has fgc by \cite{friedlander-suslin} and is power reductive owing to cocommutativity and Lemma \ref{HopfI}. 

Consider the dg $K$-module algebras $\mathcal K$ and $\mathcal K_Q=\mathcal K\otimes_Z Q$ from Lemma \ref{dgK} that are homotopically isomorphic to $\ku$ and $R$, respectively. Let $\mathcal{T} \coloneqq \REnd_{\mathcal K\otimes_Z\mathcal K^{\op}}(\mathcal K)$ and $M=\RHom_R(\ku,V)$ with $\mathcal{T}$-action given by pulling back the action of $\RHom_R(\ku)$ on $M$ along the algebra morphism
\begin{align*}
\mathfrak{def}^K_\ku\colon \mathcal{T}=\REnd_{\mathcal K\otimes_Z\mathcal K^{\op}}(\mathcal K) \to \REnd_{\mathcal K_Q}(\ku)\cong \REnd_{ R}(\ku) 
\end{align*}
consisting of the composition of \eqref{eq:action_of_S} with the induced isomorphism coming from the homotopic isomorphism between $\mathcal K_Q$ and $R$ granted by Lemma \ref{dgK} (iv). We will use Proposition \ref{intP} \ref{item:alternative5}, but in order to do so, we need to check the following requirements:
\begin{enumerate}[label=\rm{(\alph*)}]
\item $\coh(\mathcal{T})$ is affine graded-commutative and concentrated in non-negative degrees: according to Lemma \ref{dgK} (iii), $\coh(\mathcal{T})$ is homotopically isomorphic to the symmetric algebra on $\mgo/ \mgo^2$ shifted to degree $2$, which obeys the desired properties.
\item $\coh(\RHom_R(\ku,V))$ is finitely generated over $\coh(\mathcal{T})$: this follows from applying Lemma \ref{fgcR} to $M=\ku$ and $N=V$ (the requirement that $\coh(V)$ is finite-dimensional is fulfilled since $V$ is assumed finitely generated over $R\rtimes K$).
\item $\mathcal{T}$ is homotopically isomorphic to its cohomology ring $\coh(\mathcal{T})$: by Lemma \ref{dgK} (iii), the cohomology of $\mathcal T$ is $A_Z$ which in turn is homotopically isomorphic to $\mathcal{T}$ itself.
\end{enumerate}
By Proposition \ref{intP} \ref{item:alternative5}, we conclude that $\coh(\RHom_K(\ku, \RHom_R(\ku,V)))$ is Noetherian over $\coh(\RHom_K(\ku,\mathcal{T}))$, which in view of Lemma \ref{lem:NWW} means that $\coh(R\rtimes K,V)$ is Noetherian over $\coh(\RHom_K(\ku,\mathcal{T}))$.

Now, the action of $\coh(\RHom_K(\ku, \mathcal T))$ on $\coh(R\rtimes K,V)=\coh(\RHom_{R\rtimes K}(\ku,V))$ factors through that of $\coh(R\rtimes K,\ku)=\coh(\RHom_{R\rtimes K}(\ku,\ku))$, thus $\coh(R\rtimes K,V)$ is also Noetherian over $\coh(R\rtimes K,\ku)$ and, in particular, finitely generated over $\coh(R\rtimes K,\ku)$. 
When $V=\ku$, we obtain that $\coh(R\rtimes K,\ku)$ is Noetherian over itself, and thus, by Lemma \ref{lema:graded-Noetherian} \ref{item:A-f.g.}, we have that $\coh(R\rtimes K,\ku)$ is finitely generated as an algebra over $\coh^0(R\rtimes K,\ku)\cong\ku$.
\end{proof}

\begin{cor}
Let $K$ be a finite-dimensional cocommutative Hopf algebra and 
$R$ a finite-dimensional  Hopf algebra in $\yd{K}$.
If $R$ admits a $K$-equivariant deformation sequence,
then the bosonization $R \# K$ has fgc.
\end{cor}

\section{Smash products of deformation sequences} \label{sec:deformations-hopf-alg} 

\subsection{Deformation sequences of Hopf algebras} \label{subsec:deformations-hopf-alg} 
We combine the notions of deformation sequence and exact sequence of Hopf algebras in the following concept. 

\begin{definition}
\label{def:dfor-seq-hopf}
A \emph{deformation sequence of Hopf algebras} is an exact sequence of Hopf algebras (cf.~Definition \ref{ExactSequenceDef})
\begin{align*}
W \overset{\jmath}{\hookrightarrow} H \overset{\wp}{\twoheadrightarrow} K\end{align*}
which is also a deformation sequence with respect to the counit $\epsilon_W$ of $W$ (cf.~Definition \ref{def:defor-seq}). In this case, we say the Hopf algebra $K$
{\em admits a deformation 
sequence of Hopf algebras}.
\end{definition}

\vspace{1ex}

\begin{remark}
Note that a deformation sequence of Hopf algebras
is a pair of  Hopf algebra maps 
$W\overset{\jmath}{\hookrightarrow} H \overset{\wp}{\twoheadrightarrow} K$
satisfying 
\begin{itemize}
\item[\ref{item:suc-exacta-1}\ ] 
$\jmath$ is injective, and we identify $W$ with $\jmath(W)$,
\item[\ref{item:suc-exacta-2}\ ]
$\wp$ is surjective,
\item[\ref{item:suc-exacta-3}\ ] 
$\ker \wp = H\, W^+$
for $W^+=\ker \epsilon_W$,
\item[\ref{item:deformation-flat}'] 
$H$ is finitely generated as a left
$W$-module, 
\item[\ref{item:deformation-Z}\ ]
$W$ is smooth, finitely generated, and central in $H$, and
\item[\ref{item:finite-global-dim}\ ] 
$H$ has finite global dimension.
\end{itemize}
We argue that these six conditions are indeed sufficient
to imply those of Definitions~\ref{ExactSequenceDef}
and~\ref{def:defor-seq}.
Conditions \ref{item:deformation-flat}',
\ref{item:deformation-Z},
and
\ref{item:finite-global-dim}
imply that
$H$ is a finitely generated, right Noetherian PI Hopf algebra
by Remark \ref{rem:Noetherian}
and thus
faithfully flat over 
as a left module
over any Hopf subalgebra by \cite{skryabin-flatness-PI}*{Corollary 4.6},
giving condition \ref{item:deformation-flat}
of
Definition \ref{def:defor-seq}.
Condition
\ref{item:suc-exacta-4} 
of Definition~\ref{ExactSequenceDef}
 holds by Remark
\ref{rem:extension-fflat} 
and condition \ref{item:augmented-sequence} of
Definition~\ref{def:defor-seq}
holds as $\jmath$ and $\wp$ are Hopf algebra maps.
\end{remark}

\begin{remark}
When $\car \ku = 0$, a commutative Hopf algebra  is  smooth whenever it is finitely generated, 
see \cite{milneiAG}*{Theorem 3.38}.
\end{remark}

\begin{remark}\label{rem:antipode}
Let $W \hookrightarrow H \twoheadrightarrow K$ be a 
deformation sequence of Hopf algebras. It is well-known that the antipodes
of $W$ and $K$ are bijective, since $W$ is commutative and $K$ is finite-dimensional (cf. Remark~\ref{LastTermFiniteDimensional}). 
The antipode of $H$ is bijective as well
by \cite{skryabin-newresults}*{Corollary 2}
since $H$ is a Noetherian, PI, finitely generated algebra (see Remark \ref{rem:Noetherian}).
\end{remark}

\begin{remark}
The definition of 
{\em integrable Hopf algebra} in
\cite{negron-pevtsova} requires that $W$ is a coideal subalgebra;
here we assume  that $W$ is a Hopf subalgebra
in view of our applications below. 
\end{remark}

We conclude immediately from Theorem \ref{thm:negron-pevtsova}:

\begin{cor}\label{cor:Kfgc} \cite{negron-pevtsova}
A finite-dimensional  Hopf algebra 
admitting a deformation
sequence of Hopf algebras 
has  fgc. 
\end{cor}

\vspace{1ex}

Deformation sequences of Hopf algebras form a category: a {\em morphism}
$\mathbf{\Psi}: \Cgot \to \Cgot'$
from
\[\Cgot: W \overset{\jmath}{\hookrightarrow} H \overset{\wp}{\twoheadrightarrow}K \quad\text{ to }\quad \Cgot': W' \overset{\jmath'}{\hookrightarrow} H' \overset{\wp'}{\twoheadrightarrow}K'\]
is a morphism of Hopf algebras $\Psi: H \to H'$ such that $\Psi(W)  \subseteq W'$,
thus inducing a commutative diagram
\begin{align*}
\xymatrix@R-5pt@C-5pt{W\ar  @{->}[r]^{\jmath } \ar@{->}[d]
& H \ar  @{->}[r]^{\wp } \ar@{->}[d]^{\Psi } &  K\ar@{->}[d]
\\
W'\ar  @{->}[r]^{\jmath' } & H' \ar  @{->}[r]^{\wp' }&  K'.}
\end{align*}

It would be interesting to understand the subcategory of those deformation sequences of Hopf algebras with fixed $K$, where the induced map from $K$ to $K$ is the identity.

\vspace{1ex}

\subsection{Smash product of deformation sequences}\label{subsec:smash-defor-seq}
We explore the fgc for smash products by
formalizing an idea of deformation sequences
equivariant with respect to a
deformation sequence of Hopf algebras.

\begin{definition} 
\label{def:smash-product-def-seq}
Let $K$ be a Hopf algebra admitting
a deformation sequence of
Hopf algebras 
$\Cgot: W \overset{\jmath}{\hookrightarrow} H \overset{\wp}{\twoheadrightarrow} K$.
A deformation sequence of algebras 
\[\Mgot\colon Z \overset{\iota}{\hookrightarrow} Q \overset{\pi}{\twoheadrightarrow} R\]
is \emph{$\Cgot$-equivariant}
if
\medbreak
\begin{enumerate}[leftmargin=7ex,label=\rm{(\alph*)}] \setcounter{enumi}{9}
\item\label{item:equivariant-Hmod}  
$Z$, $Q$ and $R$ are augmented 
$H$-module algebras
and
$\iota$ and $\pi$ are morphisms
of augmented $H$-module algebras,

\medbreak
\item
\label{item:equivariant-W-trivial} 
$W$ acts  trivially on $Q$, and

\medbreak
\item\label{item:equivariant-Z-stable} $H$ acts trivially on $Z$.
\end{enumerate}
In this case, 
we say $R$
admits a \textit{$\Cgot$-equivariant
deformation sequence}.
\end{definition}

\vspace{1ex}

The axioms in the above definition 
allow us to define a deformation sequence
for smash products in the next result.
Note that conditions \ref{item:equivariant-W-trivial} and \ref{item:equivariant-Z-stable}
imply
that $Z\ot W$ under component-wise multiplication
is a subalgebra of $Q\rtimes H$.

\begin{prop}\label{lema:equivariant}
Let $\Cgot: W \overset{\jmath}{\hookrightarrow} H \overset{\wp}{\twoheadrightarrow} K$ be a deformation sequence of Hopf algebras and let $\Mgot: Z \overset{\iota}{\hookrightarrow} Q \overset{\pi}{\twoheadrightarrow} R$ be 
a $\Cgot$-equivariant deformation sequence. Then  

\begin{enumerate}[leftmargin=7ex,label=\rm{(\roman*)}]  
\item\label{item:equivariant-1}  $R$ is a $K$-module algebra and 
$\pi(h \cdot q) = \wp(h) \cdot \pi(q)$,  for all $h \in H$, $q \in Q$.

\medbreak
\item\label{item:equivariant-2} We have a  deformation sequence
\begin{equation} \label{eq:equiv-def-seq}
Z \otimes W \overset{\iota \rtimes \jmath}{\hookrightarrow} 
 Q \rtimes H \overset{\pi \rtimes \wp}{\twoheadrightarrow}  R \rtimes K.
\end{equation}
\end{enumerate}
\end{prop}

\begin{proof} 
For \ref{item:equivariant-1}: Let $\rho: H \to \End R$ be the associated representation. We have to check that $\rho(W^+) = 0$; then $\rho$ factors through $K$.

Choose $w \in W^+$ and $r \in R$; then $r = \pi(q)$ for some $q\in Q$
and
\begin{align*}
w\cdot r &= w\cdot \pi(q) = \pi(w\cdot q) = 0 
\end{align*}
since by condition~\ref{item:equivariant-W-trivial}
of Definition~\ref{def:smash-product-def-seq}, $w\cdot q = \epsilon(w) q = 0$.

\medbreak
For \ref{item:equivariant-2}: 
We check the conditions (a) through (h) of
Definition \ref{def:defor-seq}.
By \ref{item:equivariant-W-trivial}, $Z$ is a $W$-module subalgebra 
of $Q$, cf. Example \ref{exa:trivial-action}; hence
$\iota(w \cdot z) = \jmath(w) \cdot \iota(z)$,  for all $w \in W$, $z \in Z$.
Because of this and \ref{item:equivariant-1},  Lemma \ref{lema:smash}
implies that we have the algebra maps  $\iota \rtimes \jmath$ and $\pi \rtimes \wp$. 
We verify that they satisfy the conditions in Definition \ref{def:defor-seq}.
Evidently, \ref{item:suc-exacta-1} and \ref{item:suc-exacta-2} hold.
Next we check \ref{item:ker-pi-Z}:
\begin{align*}
\ker  (\pi \rtimes \wp)  &= \ker  \pi \otimes H + Q \otimes \ker  \wp = 
QZ^+ \otimes H + Q \otimes HW^+ 
\\
&= (Q \rtimes H) (Z^+ \otimes W + Z \otimes W^+) = (Q \rtimes H)(Z \rtimes W)^+. 
\end{align*}
Next, if $z\in Z$, $w\in W$, $q\in Q$ and $h \in H$, then
\begin{align*}
(z\otimes w) (q \otimes h)  &= z w\_{1} \cdot q \otimes w\_{2}h = zq \otimes wh,
\\
(q \otimes h)(z\otimes w)  &= q h\_{1} \cdot z \otimes h\_{2}w = qz \otimes hw,
\end{align*}
by \ref{item:equivariant-W-trivial} and \ref{item:equivariant-Z-stable}. 
Hence $z\otimes w$ is central in $Q \rtimes H$
and $Z \rtimes W = Z \otimes W$ is smooth and finitely generated, being the tensor product of two smooth and finitely generated algebras, 
thus \ref{item:deformation-Z} holds. 
Note that $Q \rtimes H$ is finite over $Z \otimes W$.

We argue that $Q \rtimes H$ is flat
over its subalgebra $Z \otimes W$
by viewing $Q\rtimes H$
as a twisted tensor product
algebra $Q\ott H$
(see \cite{CapSchichlVanzura})
with twisting map
$\tau: H\ot Q\rightarrow Q\ot H$
given by
$h\ot q\mapsto 
h\_{1}q \otimes h\_{2}$.
By Remark \ref{rem:antipode},
the antipode of $H$ is bijective
so $\tau$ is an invertible map
(e.g., see \cite{SheplerWitherspoon-TTP-AWEZ}*{Section 10}).
Corollary \ref{FlatTTPs}
with conditions \ref{item:equivariant-W-trivial} and \ref{item:equivariant-Z-stable}
implies that 
$Q \rtimes H=Q\ott H$ is flat
over $Z \otimes W=Z\rtimes W =Z\ott W$ and \ref{item:deformation-flat} holds.

Finally, by Definitions \ref{def:defor-seq} and \ref{def:dfor-seq-hopf}, 
$H$ and $Q$ have finite global dimension. 
By the main result of  \cite{LorenzLorenz}, $Q \rtimes H$ also has
finite global dimension, i.e., condition \ref{item:finite-global-dim} of Definition~\ref{def:defor-seq} holds. 
This proves that \eqref{eq:equiv-def-seq} is a deformation sequence.
\end{proof}

\vspace{1ex}

We obtain the third main result of the paper.

\begin{theorem}\label{thm:fgc-smash}
Let $K$ be a finite-dimensional Hopf algebra.
If $K$ admits a deformation sequence
$\Cgot$ of Hopf algebras
and $R$ is a finite-dimensional
algebra admiting
a $\Cgot$-equivariant deformation sequence, 
then the smash product $R \rtimes K$ has fgc.
\end{theorem}

\begin{proof} This follows from Proposition \ref{lema:equivariant} and Theorem \ref{thm:negron-pevtsova}.
\end{proof}

We immediately conclude:

\begin{cor}\label{cor:fgc-boson}
Let $K$ be a finite-dimensional Hopf algebra and $R$ a  Hopf algebra in $\yd{K}$, also finite-dimensional.
If $K$ admits a deformation sequence
$\Cgot$ of Hopf algebras
and $R$ admits
a $\Cgot$-equivariant deformation sequence, 
 then the bosonization $R \# K$ has fgc. \qed
\end{cor}

\section{Applications to Nichols algebras}
\label{sec:examples}
\medbreak

The applications we present of our main results are based on the notion of the Nichols algebra of a braided vector space. We briefly recall the main features here and refer to \cite{andrus-leyva}, \cite{andrus-angiono} or \cite{andrus-schneider} for more 
details.

\subsection{Deformation sequences of Nichols algebras}\label{subsec:nichols-examples}

 Let $(V,c)$ be a braided vector space, that is, $V$ is a vector space and the map
$c \in \text{GL}(V \otimes V)$, called the {\em braiding}, satisfies the braid equation
\[ (c \otimes \id) (\id \otimes c)(c \otimes \id) =
(\id \otimes c)(c \otimes \id) (\id \otimes c).
\]
Then the tensor algebra $T(V)$ is a braided Hopf algebra in the sense of \cite{takeuchi-braided}. The 
{\em Nichols algebra} $\toba(V) \coloneqq T(V) / \Jc(V)$ is the quotient of
$T(V)$ by a graded Hopf ideal $\Jc(V)$,  determined by the property
\begin{align*}
\Pc (\toba(V)) &= V,
\end{align*}
where $\Pc(\toba(V))$ stands for the  space of   primitive elements. 
Thus $\toba(V)$ is a braided Hopf algebra in the sense of \cite{takeuchi-braided}.
Furthermore, if $\Bc = \bigoplus_{n \geq 0} \Bc^n$
is a graded braided Hopf algebra generated by $\Bc^1$ and $\Bc^1 \cong V$ as braided vector spaces, then there are morphisms of graded braided Hopf algebras
\[T(V) \twoheadrightarrow \Bc \twoheadrightarrow \toba(V)\]
which are the identity in degree one. A braided Hopf algebra like $\Bc$ above is called
a \emph{pre-Nichols algebra}.

\medbreak
Natural examples of braided vector spaces are the Yetter-Drinfeld modules 
over a Hopf algebra $H$ and, in turn, natural examples of braided Hopf algebras 
are the  Hopf algebras in $\yd{H}$. However a given braided vector space 
may arise in many ways as a Yetter-Drinfeld module 
over $H$ for many Hopf algebras  $H$, and similarly for braided Hopf algebras.
Precisely, we introduce the following terminology for further use:

\begin{itemize}[leftmargin=7ex]\renewcommand{\labelitemi}{$\circ$}
\item A \emph{realization} of a braided vector space $V$ over a Hopf algebra $H$
consists of an action and a coaction  of $H$ on $V$ 
such that $V$ becomes an object in $\yd{H}$ and 
the braiding of $V$ in $\yd{H}$ coincides with the original one.
If such a realization exists, then we say that $V$ \emph{can be realized} over $H$.

\medbreak
\item A \emph{realization} of a braided Hopf algebra $R$ over a Hopf algebra $H$
consists of an action and a coaction  of $H$ on $R$ 
such that $R$ becomes a Hopf algebra in $\yd{H}$ with the same multiplication,
comultiplication and  braiding  in $\yd{H}$ as  the original ones.
If such a realization exists, then we say that $R$ \emph{can be realized} over $H$.
\end{itemize}

\begin{remark} 
A realization of a braided vector space $V$  over $H$ induces realizations over $H$ of 
the Hopf algebras $T(V)$ and $\toba(V)$, but not necessarily of an arbitrary 
pre-Nichols algebra  $\Bc$ of $V$.
\end{remark}

\medbreak
We will analyze examples according to the following considerations. 

\begin{prop}\label{prop:toba-fgc}
Let $V$ be a braided vector space that can be realized over a finite-dimensional Hopf algebra $K$. 
Assume that the Nichols algebra $\toba(V)$ is finite-dimensional and has fgc. 
Then $\toba(V) \# K$ has fgc, provided that

\begin{enumerate}[leftmargin=7ex,label=\rm{(\roman*)}]

\medbreak
\item $K = \Bbbk G$, where $G$ is a finite group and $\operatorname{char} \Bbbk$ does not divide $\vert G\vert$; or

\medbreak
\item $K = \Bbbk^G$, where $G$ is a finite group.
\end{enumerate}  
\end{prop}  

\pf
By Theorem \ref{th:RtoRsmashH}.
\epf

\begin{prop}\label{prop:def-seq}
Let $V$ be a braided vector space that can be realized over a finite-dimensional Hopf algebra $K$. 
Assume that the Nichols algebra $\toba(V)$ is finite-dimensional and  
that it  admits a deformation sequence 
\[\Bg\colon Z \overset{\iota}{\hookrightarrow} \Bc \overset{\pi}{\twoheadrightarrow} \toba(V)\]
where $\Bc$ is a pre-Nichols algebra of $V$. Then $\toba(V) \# K$ has fgc, provided that

\begin{enumerate}[leftmargin=7ex,label=\rm{(\roman*)}]

\item $K$ is semisimple; or

\medbreak
\item $K$ is cocommutative and $\Bg$ is $K$-equivariant; or  

\medbreak
\item\label{item:def-seq-C-equivariant} $K$ admits a deformation sequence of Hopf algebras
$\Cgot\colon W\overset{\jmath}{\hookrightarrow} H \overset{\wp}{\twoheadrightarrow} K$
such that $\Bg$ is  $\Cgot$-equivariant.  
\end{enumerate}
\end{prop}

\pf
By Theorems \ref{thm:K-semisimple-R-fgc},  \ref{theorem:negron} and \ref{thm:fgc-smash}.
\epf

\medbreak
There are several reasons to consider
this framework. First, Nichols algebras carry vital information
on the structure of pointed Hopf algebras, see \cite{andrus-schneider}.
Second, pre-Nichols algebras fitting in a deformation sequence as above appeared naturally
in the literature. Third, examples of Nichols algebras $\toba(V)$ such that 
$\toba(V) \# \ku \varGamma$ has fgc for a suitable finite abelian group $\varGamma$ 
have been studied in \cites{aapw,andruskiewitsch-natale};
but the methods in loc.~cit. do not allow us to conclude that 
$\toba(V) \# K$ has fgc for other Hopf algebras $K$.
We contribute to this problem here.

\vspace{1ex}

\subsection{The quantum line}\label{subsec:equivariant-quantum line} 
In this subsection and the next we discuss the simplest examples of Nichols algebras
as an illustration of the considerations above.

\medbreak
Let $q$ lie in  $\kut$ and let $V$ be a vector space with $\dim V = 1$
and basis $\{v\}$.
Consider
the braided vector space 
$V_q \coloneqq (V, c)$ 
with braiding 
$c \in \text{GL}(V \otimes V)$ given by $c(v \otimes v) = q \, v \otimes v$. 
It is well-known, see e.g.\ \cite{andrus-schneider}, that 
the  Nichols algebra $\toba(V_q)$ has finite dimension if and only if 
\begin{align}\label{eq:q-hyp}
q \text{ is a root of 1}, &\text{ and } q \neq 1 \text{ when } \car \ku = 0.
\end{align} 
Assume in the rest of this subsection that $q$ is as in \eqref{eq:q-hyp}. 
Let $n \in \N$ be given by
\begin{align*}
n &= \begin{cases} p &\text{if } \car \ku = p >0 \text{ and } q =1;
\\ \ord q &\text{otherwise; notice that }(p, n) = 1 \text{ if } \car \ku = p >0.
\end{cases}
\end{align*}
Then the Nichols algebra $\toba (V_q) $ is isomorphic to $\ku[X]/\langle X^n \rangle$,
where $\ku[X]$ is the polynomial ring. Calling $\iota$ for the inclusion and $\pi$ for the natural projection, we claim that
\begin{equation}\label{eq:deformation-sequence-qline}
\Mgot_q\colon \ku[X^n]\overset{\iota}{\hookrightarrow} \ku[X] \overset{\pi}{\twoheadrightarrow} \ku[X]/\langle X^n \rangle \cong \toba (V_q),   
\end{equation}
is a deformation sequence, as the seven conditions of Definition~\ref{def:defor-seq} clearly hold.

\bigbreak
Realizations of 
$V_q$ in $\yd{K}$ are described by YD-pairs, a notion already 
implicit  in \cite{andrus-schneider-p3}.

\begin{definition} \cite{andrus-leyva}*{Example 20} 
Let $K$ be a Hopf algebra with antipode $\mathcal S$. A \emph{YD-pair} for $K$ is a pair
$(g, \chi)  \in G(K) \times \hom_{\text{alg}}(K, \Bbbk)$
such that
\begin{align*}
\chi(h)\,g  &= \chi(h_{(2)}) \, h_{(1)}\, g\, \mathcal S(h_{(3)}),& \text{for any } h \in K.
\end{align*}
If $(g, \chi)$ is a YD-pair for $K$, then $g\in Z(G(K))$, the center of the group $G(K)$ of group-like elements in $K$. 
\end{definition}

\medbreak
YD-pairs classify one-dimensional objects in $\yd{K}$: 
if $(g, \chi)$ is a YD-pair, then $\Bbbk_g^{\chi} = \Bbbk$
with action and coaction given by $\chi $ and $g$ respectively, is in $\yd{K}$;
the braiding of $\Bbbk_g^{\chi}$ is multiplication by $\chi(g)$.
Hence, $\Bbbk_g^{\chi}$ realizes $V_q$ if and only if $\chi(g) = q$.

\medbreak
The following fact 
was known under the hypothesis that $K$ is either a group algebra (of dimension prime to the characteristic of the base field) or a dual group algebra as in Proposition \ref{prop:toba-fgc};
otherwise it appears to be new.

\begin{prop}\label{exa:qline-bosonK}
Let $K$ be a finite-dimensional Hopf algebra that admits a YD-pair $(g, \chi)$  such that $\chi(g) = q$, so that
$\Bbbk_g^{\chi}$ realizes $V_q$ in $\yd{K}$. If $K$ is either semisimple or
cocommutative, then $\toba(V_q) \# K$ has fgc.
\end{prop}

\pf Proposition \ref{prop:def-seq} applies, 
since $\Bbbk [X]$ and $\Bbbk [X^n]$ become Hopf algebras in $\yd{K}$ and the deformation sequence $\Mgot_q$ in \eqref{eq:deformation-sequence-qline} is $K$-equivariant. 
\epf

When $K$ is neither semisimple nor cocommutative, we still can apply 
Proposition \ref{prop:def-seq} \ref{item:def-seq-C-equivariant}, i.e., Theorem \ref{thm:fgc-smash}.

\begin{prop}\label{exa:qline-bosonK2}
Let $K$ be a finite-dimensional Hopf algebra having a YD-pair $(g, \chi)$  
such that $\chi(g) = q$, so that $\Bbbk_g^{\chi}$ realizes $V_q$ in $\yd{K}$, and $\chi^n = \varepsilon$. 
Assume that 
\begin{itemize}[leftmargin=7ex]\renewcommand{\labelitemi}{$\circ$}
\item $K$ admits a deformation sequence of Hopf algebras $\Cgot: W \overset{\jmath}{\hookrightarrow} H \overset{\wp}{\twoheadrightarrow} K$,

\item there exists $\gamma \in G(H)$ such that $\wp(\gamma) = g$ and 
$(\gamma, \chi\wp)$ is a YD-pair for $H$, and

\item $\chi \wp_{\vert W} = \varepsilon$.
\end{itemize}
Then $\toba(V_q) \# K$ has fgc.
\end{prop}

\begin{proof} We claim that the deformation sequence $\Mgot_q$ in \eqref{eq:deformation-sequence-qline} is $\Cgot$-equivariant as in Definition 
\ref{def:smash-product-def-seq}. The YD-pair $(\gamma, \chi\wp)$  
induces a realizations of $\toba(V_q)$, $\ku[X]$ and $\ku[X^n]$ as Hopf algebras in $\yd{H}$;
hence \ref{item:equivariant-Hmod} holds.
Now $H$ acts  on $X^n$ by $(\chi\wp)^n = \chi^n\wp$, so if $\chi^n= \varepsilon$, then 
\ref{item:equivariant-Z-stable} holds. The last assumption gives
\ref{item:equivariant-W-trivial}. The claim is proved, and therefore $\toba(V_q) \# K$ has fgc by
Theorem \ref{thm:fgc-smash}.
\end{proof}

\vspace{1ex}

\subsection{Quantum linear spaces}\label{subsec:equivariant-qls} 
Assume $\car \ku = p$. For $\theta\in \N$, we set $\I_{\theta} \coloneq \{1,2,\dots, \theta\}$. 
Fix a matrix $\bq = \left(q_{ij}\right)_{i, j \in \I_{\theta}}$ whose entries are roots of 1.  
Assume that
\begin{align*}
q_{ii} &\neq 1 & \text{ if } p &= 0,
\\
q_{ij}q_{ji} & = 1 &\text{ if }  i &\neq j \in \I_{\theta}.
\end{align*}
For $i \in \I_{\theta}$, set 
\begin{align}\label{eq:N_i}
N_{i} &= \begin{cases} p  &\text{if }  p >0 \text{ and } q_{ii} =1;
\\ \ord q_{ii} &\text{otherwise; so }(p, N_i) = 1 \text{ if }  p >0.
\end{cases}
\end{align}

Let $V$ be  a vector space of dimension $\theta$ with a fixed basis $v_1, \dots, v_{\theta}$.
Let $V_{\bq} \coloneqq (V, c)$ be  the braided vector space 
whose braiding  $c \in \text{GL}(V \otimes V)$ is given  by 
\begin{align}\label{eq:braiding-diagonal}
c(v_i \otimes v_j) &= q_{ij} v_j \otimes v_i & \text{for any } i,j \in \I_{\theta}.
\end{align}

\begin{lemma} The Nichols algebra $\toba(V_{\bq})$ has a presentation 
\begin{align*}
\toba(V_{\bq}) \cong \ku\langle X_1, \dots, X_{\theta} &\ \vert\   
X_{i}X_{j} -q_{ij}X_{j}X_{i}, \,  i \neq j \in \I_{\theta}; \ \ 
X_{i}^{N_i}, \, i\in \I_{\theta}  \rangle.
\end{align*} 
Therefore, the ordered monomials 
\begin{align*}
X_1^{h_1} \dots X_{\theta}^{h_{\theta}} :  \quad 
0 \leq h_i \leq N_i -1, \quad
i\in \I_{\theta},
\end{align*} 
form a basis for $\toba(V_{\bq})$ and in particular
$\dim \toba(V_{\bq}) = N_1 \dots N_{\theta}$.
\end{lemma}

As 
in \cite{andrus-schneider-p3}, we refer to $\toba(V_{\bq})$ as a quantum linear space.

\pf
This was proved in \cite{andrus-schneider-p3} for  
$p=0$ and is straightforward to establish for $p >0$.
\epf

Let $\ku_{\bq}[X_1, \dots, X_{\theta}]$ be the quotient of the free algebra $\ku\langle X_1, \dots, X_{\theta} \rangle$ by the ideal generated by $X_{i}X_{j} -q_{ij}X_{j}X_{i}$ for all $i \neq j \in \I_{\theta}$. By abuse of notation,
we denote by the same symbol the image of $X_i$ in $\ku_{\bq}[X_1, \dots, X_{\theta}]$. 
We have in $\ku_{\bq}[X_1, \dots, X_{\theta}]$ the equality
\begin{align}\label{eq:qls-powers}
X_i^{m} X_j &= q_{ij}^mX_{j}X_{i}^m, & 
i \neq j &\in \I_{\theta}, \quad m \in \N.
\end{align}
For $i\in \I_{\theta}$, we set
\begin{align*}
M_i &\coloneqq \lcm \big\{\ord q_{ij},  j \in \I_{\theta} \backslash \{i\}\big\}.
\end{align*}
Then $X_i^{M_i }$ is central in $\ku_{\bq}[X_1, \dots, X_{\theta}]$ by \eqref{eq:qls-powers}.
Recall \eqref{eq:N_i} and assume that 
\begin{align}\label{eq:qls-condition}
&M_i  \text{ divides }   N_i, && \text{for all } i \in \I_{\theta}. 
\end{align}
Then the subalgebra $Z_{\bq} = \ku\langle X_1^{N_1}, \dots, X_{\theta}^{N_{\theta}}\rangle$ of 
$\ku_{\bq}[X_1, \dots, X_{\theta}]$ is central by \eqref{eq:qls-powers}.
Hence 
\begin{equation}\label{eq:deformation-sequence-qls}
  \Mgot_{\bq} :  Z_{\bq}\hookrightarrow \ku_{\bq}[X_1, \dots, X_{\theta}] \twoheadrightarrow \toba(V_{\bq}) \cong \ku_{\bq}[X_1, \dots, X_{\theta}]/\langle X_1^{N_1}, 
\dots, X_{\theta}^{N_{\theta}} \rangle
\end{equation}
  is a deformation sequence; we recover, by Theorem \ref{thm:negron-pevtsova}, 
the well-known fact that the
quantum linear space $\toba(V_{\bq})$ has fgc, under the assumption \eqref{eq:qls-condition}. 

\medbreak
\begin{prop}\label{exa:qls-bosonK} Assume that the matrix $\bq$ satisfies
\eqref{eq:qls-condition}.
Let $K$ be a finite-dimensional Hopf algebra that admits a family
of YD-pairs  $\left(g_i, \chi_i\right)_{i \in \I_{\theta}}$ such that 
\begin{align}
 \label{eq:qls-realizationK}
\chi_j(g_i) &= q_{ij},& \text{for all }  i,j &\in \I_{\theta},
\end{align}
so that the family $\left(g_i, \chi_i\right)_{i \in \I_{\theta}}$ gives a 
realization of $V_{\bq}$ in $\yd{K}$. If $K$ is either semisimple or
cocommutative, then $\toba(V_{\bq}) \# K$ has fgc.
\end{prop}

This was known for $K$ a group algebra or a dual group algebra as in Proposition \ref{prop:toba-fgc}.

\begin{proof}
The family  $\left(g_i, \chi_i\right)_{i \in \I_{\theta}}$  induces a realization  of 
the tensor algebra $T(V_{\bq})$ as a Hopf algebra in $\yd{K}$; since the ideal generated by $X_{i}X_{j} -q_{ij}X_{j}X_{i}$ for all $i \neq j \in \I_{\theta}$ is  a subobject in $\yd{K}$, we see that
$\ku_{\bq}[X_1, \dots, X_{\theta}]$ is realized as a Hopf algebra in $\yd{K}$. 
Also $Z_{\bq} = \ku\langle X_1^{N_1}, \dots, X_{\theta}^{N_{\theta}}\rangle$ is a 
central Hopf subalgebra of $\ku_{\bq}[X_1, \dots, X_{\theta}]$ in $\yd{K}$;
i.e., $\Mgot_{\bq}$ is $K$-equivariant.
We can then apply Proposition \ref{prop:def-seq}.
\end{proof}

When $K$ is neither semisimple nor cocommutative, we still can apply 
Proposition \ref{prop:def-seq} \ref{item:def-seq-C-equivariant}, i.e., Theorem \ref{thm:fgc-smash}.

\begin{prop}\label{exa:qls-bosonK2}
Assume that the matrix $\bq$ satisfies \eqref{eq:qls-condition}.
Let $K$ be a finite-dimensional Hopf algebra that admits a family
of YD-pairs  $\left(g_i, \chi_i\right)_{i \in \I_{\theta}}$ such that 
\eqref{eq:qls-realizationK} holds and 
\begin{align}\label{eq:qls-bosonK-condition}
\chi_i^{N_i} &= \varepsilon,& \text{for all }  i &\in \I_{\theta}.
\end{align}
Assume that $K$ admits a deformation sequence of Hopf algebras $\Cgot: W \overset{\jmath}{\hookrightarrow} H \overset{\wp}{\twoheadrightarrow} K$ such that 
\begin{itemize}[leftmargin=7ex]\renewcommand{\labelitemi}{$\circ$}
\item  there exists $\gamma_{i} \in G(H)$ such that $\wp(\gamma_{i}) = g_{i}$ and 
$\left(\gamma_{i}, \chi_{i}\wp\right)$ is a YD-pair for $H$, for all $i\in \I_{\theta}$, and 

\item $\chi_i \wp_{\vert W} = \varepsilon$ for all $i\in \I_{\theta}$.
\end{itemize}
Then $\toba(V_{\bq}) \# K$ has fgc.
\end{prop}

\begin{proof} As in the proof of Proposition \ref{exa:qls-bosonK}, the family $\left(g_i, \chi_i\right)_{i \in \I_{\theta}}$ gives a 
realization of $V_{\bq}$ in $\yd{K}$. As argued above, 
$\Mgot_{\bq}$ in \eqref{eq:deformation-sequence-qls} is a deformation sequence because of 
\eqref{eq:qls-condition}.
We claim that $\Mgot_{\bq}$ is $\Cgot$-equivariant as in Definition 
\ref{def:smash-product-def-seq}. Since $\chi_{j}\wp (\gamma_i) = \chi_{j}(g_i) = q_{ij}$ for all 
$i, j\in \I_{\theta}$, 
the family $(\gamma_i, \chi_i\wp)$  of YD-pairs
induces  realizations of $\toba(V_{\bq})$, $\ku_{\bq}[X_1, \dots, X_{\theta}]$ 
and $Z_{\bq}$ as Hopf algebras in $\yd{H}$;
hence \ref{item:equivariant-Hmod} holds.

Now $H$ acts  on $X_{i}^{N_i}$ by $(\chi_i\wp)^{N_i} = \chi_i^{N_i}\wp$; then 
\ref{item:equivariant-Z-stable} holds by \eqref{eq:qls-bosonK-condition}. The last assumption gives
\ref{item:equivariant-W-trivial}. The claim is proved, hence $\toba(V_{\bq}) \# K$ has fgc by
Theorem \ref{thm:fgc-smash}.
\end{proof}

\vspace{1ex}

\subsection{Nichols algebras of diagonal type}\label{subsec:equivariant-cartan type}
We discuss informally how the methods exposed here apply to finite-dimensional 
Nichols algebras of diagonal type. See, e.g., the surveys \cites{andrus-leyva,andrus-angiono}
for detailed expositions, unexplained terminology and more references.

\medbreak
 In this subsection, set $\car \ku = 0$. Let $\theta\in \N$ and 
fix a matrix $\bq = \left(q_{ij}\right)_{i, j \in \I_{\theta}}$ whose entries are roots of 1.
We assume that the associated Dynkin diagram is connected.
Let $V$ be  a $\ku$-vector space of dimension $\theta$ with a fixed basis $\{v_1, \dots, v_{\theta}\}$.
Let $V_{\bq} \coloneqq (V, c)$ be the braided vector space 
whose braiding  $c \in \text{GL}(V \otimes V)$ is given  by \eqref{eq:braiding-diagonal}.
We summarize the main features of the theory:

\medbreak
\begin{itemize}  [leftmargin=7ex]
\item The classification of the matrices $\bq$ such that $\dim \toba(V_{\bq}) < \infty$
was achieved in \cite{heckenberger}. 
In \cite{andrus-angiono} an organization of the classification into several types was proposed; these types are:
\end{itemize}

\begin{itemize} [leftmargin=12ex] \renewcommand{\labelitemi}{$\circ$}
\medbreak
\item  Cartan type, related to finite-dimensional simple Lie algebras in $\car 0$.

\medbreak
\item Super type, related to finite-dimensional simple contragredient  Lie superalgebras in $\car 0$. 

\medbreak
\item Modular type, related to finite-dimensional simple Lie algebras or superalgebras in $\car> 0$,
not in the previous classes. 

\medbreak
\item UFO type, 12 examples not known to be related to Lie theory (yet).
\end{itemize} 

\medbreak
\begin{itemize}  [leftmargin=7ex]
\item The defining relations of the finite-dimensional $\toba(V_{\bq})$ were described in
\cite{angiono-crelle}. Motivated by some arguments in the proof and by the theory of quantum groups
developed by De Concini and Procesi, a \emph{distinguished} pre-Nichols algebra 
$\widetilde{\toba}(V_{\bq})$ was introduced in \cite{angiono-transfgps}. Furthermore, 
$\widetilde{\toba}(V_{\bq})$ has a normal braided Hopf subalgebra $Z^+(V_{\bq})$, see 
 \cite{angiono-transfgps}*{Theorem 31}. 

 \medbreak
 \item  We follow now the exposition in \cite{aay}*{Subsection 4.5}. Assume that the condition
 \cite{aay}*{Condition (4.26)} holds; notice that this extends \eqref{eq:qls-bosonK-condition}. 
 Then $Z^+(V_{\bq})$ is central in $\widetilde{\toba}(V_{\bq})$ and we have an exact sequence 
 \begin{equation}
 \label{defdiagonal}
Z^+(V_{\bq}) \overset{\iota}{\hookrightarrow} \widetilde{\toba}(V_{\bq}) \overset{\pi}{\twoheadrightarrow} \toba(V_{\bq})
\end{equation}
of braided Hopf algebras, see \cite{andrus-natale-braided} for this notion.
Now, for \eqref{defdiagonal} to be a deformation sequence, 
we need $\widetilde{\toba}(V_{\bq})$ to have finite global dimension.
If $\bq$ is of Cartan type, then $\widetilde{\toba}(V_{\bq})$
has finite global dimension by standard arguments, and one can check that \eqref{defdiagonal} is a deformation sequence. Otherwise, one of the degree 1 generators of $\widetilde{\toba}(V_{\bq})$ is nilpotent, which prevents $\widetilde{\toba}(V_{\bq})$ from having finite global dimension.

\medbreak
\item If $Z \overset{\iota}{\hookrightarrow} \Bc \overset{\pi}{\twoheadrightarrow} \toba(V_{\bq})$
is a deformation sequence where $\Bc$ is a pre-Nichols algebra of $V_{\bq}$, then
$\GK \Bc < \infty$ by Remark \ref{rem:Noetherian}. Up to a short list of exceptions, it is known that $\Bc$ is a quotient of $\widetilde{\toba}(V_{\bq})$, see \cites{andrus-sanmarco,angiono-etal} 
and references therein.
Therefore, when $\bq$ is not of Cartan type, $\Bc$ has infinite global dimension by the same argument as above. In conclusion, it appears that $\toba(V_{\bq})$ has a deformation sequence
only when $\bq$ is of Cartan type, up to some mild conditions on $\bq$.

\medbreak
\item By a different technique, it was shown in \cite{aapw} that $\toba(V_{\bq})$ has fgc
when $\bq$ "belongs to a family",
which leaves open only 27 examples of modular or UFO type.
\end{itemize}

\medbreak
In conclusion, the methods of the present paper apply in Cartan type. 
See \cite{aay}*{Subsection 4.5} for details of the following arguments. 
We need to introduce some notation.
Fix a matrix $\bq$ of Cartan type, i.~e., there exists a finite Cartan matrix 
$\mathbf{a} = (a_{ij})_{i, j \in \I_{\theta}}$ such that $q_{ij}q_{ji} = q_{ii}^{a_{ij}}$ for all 
$i \neq j \in \I_{\theta}$. 
Let $\varDelta_+$ be the set of positive roots corresponding to $\mathbf{a}$
with subset of simple roots $\{\alpha_{i}: i \in \I_{\theta}\}$. 
For $\alpha = \sum_{i \in \I_{\theta}} a_i \alpha_i$, $\beta = \sum_{i \in \I_{\theta}} b_i \alpha_i \in \varDelta_+$, we set:
\[
q_{\alpha\beta} = \prod_{i, j \in \I_{\theta}} q_{ij}^{a_ib_j} \qquad \text{ and } \qquad 
N_{\beta} = \ord q_{\beta\beta}.
\]
As explained in \cite{aay}, there exists $X_{\beta} \in \widetilde{\toba}(V_{\bq})$ (called $e_{\beta}$
in \emph{loc. cit.}), defined using the Weyl group automorphisms and which is in fact an
iterated braided commutator.
Then
\begin{align*}
Z^+(V_{\bq}) &= \ku\langle X_{\beta}^{N_{\beta}}: \beta \in  \varDelta_+\rangle.
\end{align*}
For $Z^+(V_{\bq})$ to be central in $\widetilde{\toba}(V_{\bq})$ we need 
to assume, cf. \cite{aay}*{Condition (4.26)}:
\begin{align}\label{eq:cartan-centrality}
q_{\alpha\beta}^{N_{\beta}} &= 1, & \text{for all } \alpha, \beta \in  \varDelta_+.
\end{align}
\begin{prop}\label{exa:cartan-bosonK}
Let $\bq$ be a matrix of Cartan type
satisfying \eqref{eq:cartan-centrality}. 
Let $K$ be a finite-dimensional Hopf algebra that admits a family
of YD-pairs  $\left(g_i, \chi_i\right)_{i \in \I_{\theta}}$ such that 
\eqref{eq:qls-realizationK} holds, so that it gives a 
realization of $V_{\bq}$ in $\yd{K}$. 
Assume that either $K$ is semisimple; or $K$ is cocommutative and \eqref{defdiagonal} is $K$-equivariant.
Then $\toba(V_{\bq}) \# K$ has fgc.
\end{prop}

For the case when $K$ is a group algebra or a dual group algebra, see \cite{aapw}
and \cite{MPSW}.

\pf As discussed above, $\gldim \widetilde{\toba}(V_{\bq}) < \infty$ in this case, so \eqref{defdiagonal} is a deformation sequence and Proposition \ref{prop:def-seq} applies.
\epf

When $K$ is neither semisimple nor cocommutative, we  can apply 
Proposition \ref{prop:def-seq} \ref{item:def-seq-C-equivariant}.

\begin{prop}\label{exa:cartan-bosonK2}
Let $\bq$ be a matrix of Cartan type satisfying \eqref{eq:cartan-centrality}.
Let $K$ be a finite-dimensional Hopf algebra admitting a family of YD pairs $\left(g_i, \chi_i\right)_{i \in \I_{\theta}}$ such that
\eqref{eq:qls-realizationK} and
\begin{align}\label{eq:cartan-bosonK-condition}
\chi_{\beta}^{N_{\beta}} &= \varepsilon,& \text{for all }  \beta &\in \varDelta_+,
\end{align}
hold, where for any $\beta = \sum_{i \in \I_{\theta}} b_i \alpha_i  \in \varDelta_+$, we set
\begin{align*}
\chi_{\beta} &= \prod_{i \in \I_{\theta}} q_{ij}^{b_i}.
\end{align*}

Assume that $K$ has a deformation sequence of Hopf algebras $\Cgot: W \overset{\jmath}{\hookrightarrow} H \overset{\wp}{\twoheadrightarrow} K$ such that 
\begin{itemize}[leftmargin=7ex]\renewcommand{\labelitemi}{$\circ$}
\item  there exists $\gamma_{i} \in G(H)$ such that $\wp(\gamma_{i}) = g_{i}$ and 
$\left(\gamma_{i}, \chi_{i}\wp\right)$ is a YD-pair for $H$, for all $i\in \I_{\theta}$, and 

\item $\chi_i \wp_{\vert W} = \varepsilon$ for all $i\in \I_{\theta}$.
\end{itemize}
Then $\toba(V_{\bq}) \# K$ has fgc.
\end{prop}

\pf The proof is analogous to that of Proposition \ref{exa:qls-bosonK2}, noting that
$H$ acts  on $X_{\beta}^{N_{\beta}}$ by $(\chi_{\beta}\wp)^{N_{\beta}} = \chi_{\beta}^{N_{\beta}}\wp$; then 
\ref{item:equivariant-Z-stable} holds by \eqref{eq:cartan-bosonK-condition}.
\epf

\vspace{1ex}

\subsection{The restricted Jordan plane}\label{subsec:jordan} 
In this subsection, assume $\car \ku = p$ is an odd prime. 
Let  $\Vc(1,2)$ be the 2-dimensional braided vector space
with a basis $\{x,y\}$ and the braiding determined by
\begin{align}\label{equation:basis-block}
\begin{aligned}
c(x \ot  x) &=  x \ot  x,&c(y \ot  x) &=  x \ot  y, \\
c(x \ot  y) &=(y+ x) \ot  x,& c(y \ot  y) &=(y+x) \ot  y.
\end{aligned}
\end{align}
It was shown in \cite{clw} that the Nichols algebra $\toba\left(\Vc(1,2) \right)$, 
also called \emph{the restricted Jordan plane}, is presented as 
\begin{align*}
\toba\left(\Vc(1,2) \right)  \cong \ku\langle x, y\ \vert\  yx-xy+\tfrac{1}{2}x^2,
\ x^p,  \ y^p  \rangle.
\end{align*}
The  \emph{Jordan plane} is the algebra $J =\ku\langle x, y\ \vert\  yx-xy+\tfrac{1}{2}x^2  \rangle$. 
From \cite{clw}*{Lemma 3.8}, 
\begin{align*}
x^ny &=yx^n +\tfrac{n}{2}x^{n+1},
\\
y^n x &= \sum_{k=0}^{n} \binom{n}{k} \frac{(-1)^k k!}{2^k} x^{k+1} y^{n -k}. 
\end{align*}
We conclude that $Z_J \coloneqq\ku \langle x^p, y^p\rangle$ is a central subalgebra of $J$.
Actually, $Z_J$ is a braided Hopf subalgebra  of $J$, cf. \cite{clw}*{Lemma 3.8} and we have a  deformation sequence
\begin{align*}
\Jgot : Z_{J} \hookrightarrow J
\twoheadrightarrow  \toba\left(\Vc(1,2) \right).
\end{align*}
We recover the fact that $\toba\left(\Vc(1,2) \right)$ has fgc \cite{NWW2019}. We turn to
realizations of the restricted Jordan plane 
that can be obtained as follows.

\begin{definition} \cite{andrus-leyva}*{Example 23}
Let $K$ be a Hopf algebra with antipode $\mathcal S$. 
 A \emph{YD-triple} for $K$ is a collection $(g, \chi, \eta)$ where
$(g, \chi)$ is a YD-pair for $K$ and $\eta \in \Der_{\chi,\chi}(K, \ku)$
such that
\begin{align}\label{eq:YD-triple}
&&\eta(h) g &= \eta(h\_2) h\_1 g \Ss(h\_3), & \text{for any } h \in K,
\\ \label{eq:YD-triple-jordan}
&&\chi(g) &= \eta(g) = 1.
\end{align}
\end{definition}

\medbreak
A YD-triple $(g, \chi, \eta)$ gives rise to $\Vc_g(\chi,\eta) \in \yd{K}$, defined as the $\ku$-vector space with a basis $\{x,y\}$, whose $K$-action and $K$-coaction are respectively given by
\begin{align*}
h\cdot x &= \chi(h) x,& h\cdot y&=\chi(h) y + \eta(h)x,&h&\in K;& 
\delta(x) &= g\otimes x,& \delta(y) &= g\otimes y.
\end{align*}
By assumption \eqref{eq:YD-triple-jordan}, 
$\Vc_g(\chi, \eta)\cong \Vc(1,2)$ as a braided vector space, where the braiding of $\Vc_g(\chi, \eta)$ in $\yd{K}$ coincides with that of $\Vc(1,2)$ given in \eqref{equation:basis-block}.

It was observed that a finite-dimensional Hopf algebra that admits a YD-triple $(g, \chi, \eta)$
is not semisimple \cite{andruskiewitsch-natale}*{Remark 4.8}.
For cocommutative Hopf algebras, 
Proposition \ref{prop:def-seq}
implies the following fact.

\begin{prop} \label{exa:jordan}
Let $K$ be a finite-dimensional cocommutative Hopf algebra that admits
a YD-triple $(g, \chi, \eta)$   for $K$. Then 
$\toba \left(\Vc_g(\chi, \eta)\right) \# K$ has fgc. \qed
\end{prop}

\begin{example} \cite{andruskiewitsch-natale}*{Remark 4.9}
Let $G$ be a finite group and let $(g, \chi, \eta)$ be a YD-triple for $\ku G$,
so that $g \in Z(G)$ and $\chi \in \Hom_{\rm gps} (G, \ku^{\times})$.
Then the Hopf algebra $H  = \toba(\Vc_g(\chi, \eta)) \# \ku G$ 
fits into the abelian exact sequence 
\begin{align}\label{eq:boso-jordan-gral}
\ku \to K \overset{\iota}{\rightarrow} H \overset{\pi}{\rightarrow} L \to \ku 
\end{align}
defined as follows:
\begin{itemize}[leftmargin=7ex]\renewcommand{\labelitemi}{$\circ$} 

\item $K \coloneqq  \ku \langle x, \gamma: \gamma \in N\rangle 
\simeq \ku\langle x \rangle\# \ku N$, where $N \coloneqq \ker \chi \cap Z(G)\lhd F$;
\item $L \coloneqq \toba(\zeta) \# \ku (G/N) \simeq \ku[\zeta]/ (\zeta^p) \# \ku (G/N)$,
where $\ku \zeta \in \yd{\ku (G/N)}$ is given by the YD-pair $(e, \overline{\chi})$
and $\zeta$ is primitive;
\item $\iota$ is the inclusion and  $\pi$ is 
given by $\pi(x) = 0$, $\pi(y) = \zeta$ and $\pi(\gamma) =$  $\gamma N$ in $G/N$.
\end{itemize}
Observe that \eqref{eq:boso-jordan-gral} is not necessarily split, e.g., when
the exact sequence of groups $1 \to N \rightarrow G \rightarrow G/N \to 1$ is not split.

When $G = \Z^p$, $H$ has fgc by \cite{NWW2019} or by \cite{andruskiewitsch-natale}, 
with different proofs. Proposition \ref{exa:jordan} implies that $H$ has fgc for any $G$.
This fact also confirms that 
\cite{andruskiewitsch-natale}*{Question 1.1}, namely that
an extension of finite-dimensional Hopf algebras has fgc whenever the kernel and the cokernel have fgc,
has a positive answer.
\end{example}  

\vspace{1ex}

\subsection{Sums of blocks and points}\label{subsec:blocks-points} 
In this subsection, assume $\car \ku = p$ is an odd prime.
We discuss informally a class of braided vector spaces, introduced in \cite{aah-oddchar},
which decompose as direct sums of Jordan blocks, super Jordan blocks and labeled points.
Their Nichols algebras are finite-dimensional. 

\medbreak
Precisely, for $\theta\in \N$, we consider the subclass  $\Vs_+$ consisting of braided vector spaces which decompose 
as a direct sum of $\theta$ summands, $t >0$ of them being Jordan blocks and  $\theta - t$ of them being points labeled with $1$ (i.e.,  one-dimensional braided vector spaces with the usual 
transposition) plus some technical conditions, see \cite{andruskiewitsch-natale}*{Section 6} for details. 
For fixed $t$ and $\theta$, the elements of $\Vs_+$ depend on two parameters:

\medbreak
\begin{itemize}  [leftmargin=6ex]\renewcommand{\labelitemi}{$\circ$}
\item a matrix $\bq = (q_{ij})_{i,j \in \I_{\theta}}$ such that 
\begin{align*}
q_{ij}q_{ji} &= 1, & q_{ii} &= 1, & \text{for all } i, j &\in \I_{\theta},\,  i \neq  j;
\end{align*}

\medbreak
\item a family  $\ba = (a_{ij}) _{\substack{i \in \I_{t + 1, \theta}, \\  j \in  \I_{t}}}$
with entries in $\fp$, 
\end{itemize}
where for $k < \ell \in \N$, we set $I_{k,\ell} = \{k, k + 1, . . . , \ell\}$, and $I_\ell = I_{1,\ell}$.

Let $\Vs(\bq, \mathbf{a})$ be the braided vector space corresponding to
$(\bq, \mathbf{a})$. 
By  \cite{andruskiewitsch-natale}*{Lemma 6.3}, 
the Nichols algebra $\toba(\Vs(\bq, \ba))$ is presented by generators 
$x_i$ and $y_j$, for $i\in \I_{\theta}, j\in \I_{t}$, and relations
(45), \dots, (50) in \emph{loc.\  cit}. Let $\Bc$ be the algebra presented by generators 
$x_i$ and $y_j$, for $i\in \I_{\theta}, j\in \I_{t}$, and relations
\begin{align}
\tag{45 ii}
y_jx_j -x_jy_j+\tfrac{1}{2}x_j^2 &=0, & \text{for all } j &\in\I_{t},
\end{align}
that is, the second part of (45), (46), (47), (48), (50) in \emph{loc.\ cit}.
The analogue of $\Bc$ in characteristic 0 was studied in \cite{aah-memoirs};
indeed, it is a Nichols algebra in that setting.
Arguing as in \cite{aah-memoirs}, we see that $\Bc$ is a pre-Nichols algebra that fits into
a deformation sequence
\begin{align*}
\Bg \colon Z \hookrightarrow \Bc
\twoheadrightarrow  \toba\left(\Vs(\bq, \mathbf{a}) \right).
\end{align*}

\begin{prop} \label{exa:jordan+points}
Let $K$ be a finite-dimensional cocommutative Hopf algebra such that $\Vs(\bq, \mathbf{a})$ admits
a realization over $K$. Then 
$\toba \left(\Vs(\bq, \mathbf{a})\right) \# K$ has fgc.
\end{prop}

\pf By Proposition \ref{prop:def-seq}. \epf

This was proved also in \cite{andruskiewitsch-natale}*{Theorem 6.7} by the following arguments:

\medbreak
\begin{itemize}  [leftmargin=*]\renewcommand{\labelitemi}{$\diamond$}
\item Let $\uno$ be the matrix with all entries equal to $1$. 
Then there exists a suitable finite abelian 
group $\varGamma$ such that $\toba \left(\Vs(\uno, \mathbf{a})\right)\# \Bbbk \varGamma$
fits into a split exact sequence 
\begin{align*}
\ku \to W \overset{\iota}{\rightarrow} \toba \left(\Vs(\uno, \mathbf{a})\right)\# \Bbbk \varGamma \overset{\pi}{\rightarrow} \ugo(\lgo) \to \ku
\end{align*}
where $W$ is a commutative Hopf subalgebra and $\lgo$ is a restricted Lie algebra.

\medbreak
\item  Then the Drinfeld double 
$D\left(\toba \left(\Vs(\uno, \mathbf{a})\right)\# \Bbbk \varGamma\right)$
has fgc by \cite{andruskiewitsch-natale}*{Theorem 3.6}. Consequently, 
$\toba \left(\Vs(\uno, \mathbf{a})\right)\# \Bbbk \varGamma$ and 
$\toba \left(\Vs(\uno, \mathbf{a})\right)$ have fgc.

\medbreak
\item For the same  group $\varGamma$, 
$\toba(\Vs(\bq, \ba))\# \ku \Gamma$ 
is a cocycle deformation of $\toba(\Vs(\uno, \ba))\# \ku \Gamma$,
cf.  \cite{andruskiewitsch-natale}*{Corollary A8}. Then
$D\left(\toba \left(\Vs(\bq, \mathbf{a})\right)\# \Bbbk \varGamma\right)$
is isomorphic to
$D\left(\toba \left(\Vs(\uno, \mathbf{a})\right)\# \Bbbk \varGamma\right)$
as algebras, hence it has fgc; a fortiori,
$\toba \left(\Vs(\bq, \mathbf{a})\right)\# \Bbbk \varGamma$ and 
$\toba \left(\Vs(\bq, \mathbf{a})\right)$ have fgc.
\end{itemize}

Now Proposition \ref{exa:jordan+points} says that replacing $\Bbbk \varGamma$ by a 
cocommutative $K$, in particular by $\Bbbk G$ where $G$ is a finite group, 
such that $\Vs(\bq, \mathbf{a})$ is realizable over $K$, we still have that
$\toba \left(\Vs(\bq, \mathbf{a})\right)\# K$
has fgc.
It can be shown that this
Hopf algebra fits into a (not necessarily quasi-split) abelian extension
giving further evidence for a positive answer to \cite{andruskiewitsch-natale}*{Question 1.1}.
On the opposite side, the method of the present paper does not establish that 
$D\left(\toba \left(\Vs(\bq, \mathbf{a})\right)\# \Bbbk \varGamma\right)$ 
has fgc, a fact that remains an open question.

\vspace{1ex}

\subsection*{Funding}
The work of N.~Andruskiewitsch was  partially supported by Secyt (UNC), CONICET (PIP 11220200102916CO), and FONCyT-ANPCyT (PICT-2019-03660). D.~Jaklitsch was supported by the Research Council of Norway - project 324944. V.~Nguyen was partially supported by US NSF grant DMS-2201146. J.~Plavnik was partially supported by US NSF Grant DMS-2146392 and by Simons Foundation Award \#889000 as part of the Simons Collaboration on Global Categorical Symmetries. A.~Shepler was partially supported by Simons Foundation grant \#429539. 

\vspace{1ex}

\subsection*{Disclaimer} The views expressed in this article are those of the author(s) and do not reflect the official policy or position of the U.S. Naval Academy, Department of the Navy, the Department of Defense, or the U.S. Government.  

\vspace{1ex}

\subsection*{Acknowledgments} This project began while the authors were participating in a workshop at the American Institute of
Mathematics, whose hospitality and support are gratefully acknowledged. We thank Cris Negron, Julia Pevtsova, Serge Skryabin, and Sarah Witherspoon for helpful conversations.

\appendix
\section{Flatness for twisted tensor products}
\label{Appendix}

The semidirect product algebras we consider in this paper are examples of twisted tensor products.
Recall that a {\em twisted tensor product
algebra} $Q\ott H$
of two algebras $Q$ and $H$ is the vector space $Q\ot H$
endowed with an algebra structure given by
a {\em twisting map}
$\tau: H\ot Q\rightarrow Q\ot H$,
see \cite{CapSchichlVanzura}.
We use the next results on twisted
tensor products in the proof of Proposition \ref{lema:equivariant}.

\begin{lemma}\label{OneSidedTTP}
  Suppose $Q\ott H$ is a twisted tensor product algebra
  defined by 
  an invertible twisting map $\tau:H\ot Q\rightarrow Q\ot H$.
  If $Q$ is flat over a subalgebra $Q'$
  and $\tau$ restricts to a bijection $H\ot Q'\rightarrow Q'\ot H$,
  then
$Q\ott H$ is flat over
  $Q'\ott H$.
  \end{lemma}
  \begin{proof}
For any right module $M$ over $Q'\ott H$,
consider the inherited right actions on $M$ by $Q'$ and $H$
obtained by identifying
$Q'\cong Q'\ot \ku 1_{H}$ 
and $H\cong \ku 1_{Q}\ot H$ with vector subspaces of $Q\ot H\cong Q\ott H$.
Note that $H\ot Q$ itself  is a left $Q'$-module
via the map
\begin{equation}\label{FirstAction}
(\id_{H}\ot\,  m_Q)(\tau^{-1}\ot \id _Q) :\ \
Q'\ot (H\ot Q)\rightarrow H\ot Q
\end{equation}
for $m_Q:Q\ot Q\rightarrow Q$ the multiplication map on $Q$.
As $\tau^{-1}(Q' \ot H)\subset H\ot Q'$,
we may endow
$M\ot_{Q'} Q$ with the structure of a right $H$-module 
via the composition
\begin{equation}\label{SecondAction}
  (M\ot_{Q'} Q)\ot H
  \  \xrightarrow{\ \id_M\ot \, \tau^{-1}\ } \
M\ot_{Q'} (H\ot Q)
\  \xrightarrow{ \ \rho\, \ot\, \id _Q\ }\
M\ot_{Q'} Q
  \end{equation}
  for $\rho: M\ot H\rightarrow H$
  the right $H$-module structure map on $M$
  using the fact that twisting commutes with the multiplication in
  $Q$ and in $H$, see \cite{CapSchichlVanzura}.

  We define an abelian group isomorphism   \begin{equation}
  \label{TensoringOverTwistedProduct}
M\ot_{Q'} Q\cong (M\ot_{Q'} Q) \ot_{H} H
\ \xrightarrow{ \id _M\ot \id _Q\ot \id _K}
M\ot_{(Q'\ott H)} (Q \ott H),
\end{equation}
with inverse map given by $ \id _M\ot \id _Q\ot \id _H$,
using  the action in  
(\ref{SecondAction}) and the
 fact that $\tau$ merely
swaps tensor
factors on any input with a tensor component from $k$.

To see that $Q\ott H$ is flat over $Q'\ott H$,
take any injective map $f:M_1\rightarrow M_2$
of right $(Q'\ott H)$-modules.
As $Q$ is flat over $Q'$, this extends to an injective map
$f\ot \id : M_1\ot_{Q'} Q\rightarrow M_2\ot_{Q'} Q$
which is also a map of right $H$-modules
under the action of (\ref{SecondAction})
as $f$ is a right $H$-module homomorphism.
We obtain a commutative diagram
using (\ref{TensoringOverTwistedProduct}):
$$
\begin{tikzcd}
 (M_1\ot_{Q'} Q) \ot_H H
  \arrow[rr, "\ \  f\ot \id \ot \id \ \ "]
  &&  (M_2\ot_{Q'} Q) \ot_H H
   \arrow[d, swap,  " \id  \ot \id  \ot \id  "]
 \\
 M_1\ot_{(Q'\ott H)}  (Q\ott H)
 \arrow[u,  "\id  \ot \id  \ot \id "]
\arrow[rr, swap, "\ \  f\ot (\id \ot \id )\ \ "]
&& M_2\ot_{(Q'\ott H)} (Q\ott H)
\, .
\end{tikzcd}
$$
As the top row map 
  is injective, so too is the bottom row map.
\end{proof}

\begin{cor}\label{FlatTTPs}
 Suppose algebras $Q$ and $H$ are flat over subalgebras $Q'$ and $H'$, respectively, and $Q\ott H$ is a twisted tensor product algebra
  defined by 
  a twisting map $\tau:H\ot Q\rightarrow Q\ot H$
  that restricts to bijections $H\ot Q'\rightarrow Q'\ot H$
  and $H'\ot Q'\rightarrow Q'\ot H'$.
  Then $Q\ott H$ is flat over
  $Q'\ott H'$.
\end{cor}
\begin{proof}
  By Lemma \ref{OneSidedTTP},
  $Q\ott H$ is flat over $Q'\ott H$
  and $Q'\ott H$ is flat over $Q'\ott H'$,
  hence $Q\ott H$ is flat over $Q'\ott H'$.
\end{proof}

\vspace{1ex}

\bibliographystyle{amsrefs}
\bibliography{refs-boso}

\end{document}